\documentclass[a4,12pt,leqno]{amsart}
\usepackage{amsmath}
\usepackage{amsfonts}
\usepackage{amssymb}
\usepackage{mathrsfs}
\usepackage{sidecap}
\usepackage{float}
\usepackage{graphicx}
\usepackage{cite}
\usepackage{placeins}
\restylefloat{figure}

\newtheorem{theorem}{Theorem}
\newtheorem{acknowledgement}[theorem]{Acknowledgement}
\newtheorem{algorithm}[theorem]{Algorithm}

\newtheorem{lemma}[theorem]{Lemma}
\newtheorem{notation}[theorem]{Notational Convention}

\newtheorem{remark}[theorem]{Remark}

\begin{document}

\title[Rate of convergence]{Rate of convergence: the packing and centered Hausdorff measures of
totally disconnected self-similar sets}

\author{Marta Llorente}
\address{Marta Llorente: Departamento de
 An\'{a}lisis Econ\'{o}mico: Econom\'{i}a Cuantitativa\\Universidad Aut\`{o}noma de Madrid, Campus de Cantoblanco, 28049 Madrid \\  Spain\\ }
\email{m.llorente@uam.es }

\author{M. Eugenia Mera}
\address{M${{}^a}$ Eugenia Mera: Departamento An\'{a}lisis Econ\'{o}mico I\\
Universidad Complutense de Madrid \\
Campus de Somosaguas, 28223 Madrid, Spain\\}
\email{mera@ccee.ucm.es }
\author{Manuel Mor\'{a}n}
\address{Manuel Mor\'{a}n: Departamento An\'{a}lisis Econ\'{o}mico I\\
Universidad Complutense de Madrid \\
Campus de Somosaguas, 28223 Madrid, Spain\\}
\email{mmoranca@ccee.ucm.es }

\thanks{This research was supported by the Complutense University of Madrid and
Santander Universidades, research project  940038.}
\date{}

\maketitle

\begin{abstract}
In this paper we obtain the rates of convergence of the algorithms
given in \cite{[LLM3]} and \cite{[LLM4]}  for an automatic
computation of the centered Hausdorff and packing measures of a
totally disconnected self-similar set. We evaluate these rates
empirically through the numerical analysis of three standard
classes of self-similar sets, namely, the families of Cantor type
sets in the real line and the plane and the class of Sierpinski
gaskets. For these three classes and for small contraction ratios,
sharp bounds for the exact values of the corresponding measures
are obtained and it is shown how these bounds automatically yield
estimates of the corresponding measures, accurate in some cases to
as many as 14 decimal places. In particular, the algorithms
accurately recover the exact values of the measures in all cases
in which these values are known by geometrical arguments. Positive
results, which confirm some conjectural values  given in
\cite{[LLM3]} and \cite{[LLM4]} for the measures, are also
obtained for an intermediate range of larger contraction ratios.
We give an argument showing that, for this range of contraction
ratios, the problem is inherently computational in the sense that
any theoretical proof, such as those mentioned above, might be
impossible, so that in these cases, our method is the only
available approach. For contraction ratios close to those of the
connected case our computational method becomes intractably time
consuming, so the computation of the exact values of the packing
and centered Hausdorff measures in the general case, with the open
set condition, remains a challenging problem.
\end{abstract}

\section{Introduction\label{intro}}

The present paper is part of a program aimed at finding a method for the
automatic computation of metric measures, such as the packing or Hausdorff
measure, of a given fractal set. In particular, we obtain the rates of
convergence of the algorithms given in \cite{[LLM3]} and \cite{[LLM4]} for
computing the centered Hausdorff and packing measures, respectively, of a
totally disconnected self-similar set. It is important to note that,
although the convergence of these algorithms was shown in \cite{[LLM3]} and
\cite{[LLM4]} without establishing their rates of convergence, the outputs
of the algorithms were still useful for obtaining conjectural values of the
measures. Using results presented in this note we can prove these
conjectures (see Sections \ref{packex} and \ref{HCex}).

Recall that a totally disconnected self-similar set associated to a system $%
\Psi =\{f_{1,}f_{2,\dots ,}f_{m}\}$ of contracting similitudes of $\mathbb{R}%
^{n}$ is a compact non-empty set $E$ $\subset\mathbb{R}^{n}$\ such that $%
E=\bigcup_{i=1}^{m}f_{i}(E)$ and satisfying

\begin{equation}
f_{i}(E)\cap f_{j}(E)=\emptyset \,\,\forall \,\,i\neq j,\qquad i,j\in
\{1,\dots ,m\}=:M.  \label{SSC}
\end{equation}
The last condition implies the \textit{Open Set Condition (OSC, see \cite%
{[HUTCH]})}, and it is known as the \emph{Strong Separation Condition (SSC)}%
. Throughout the paper we assume the system $\Psi $ satisfies the \emph{SSC}
and write
\begin{equation}
c:=\min_{\substack{ i,j\in M  \\ i\neq j}}d_{\inf }(f_{i}(E),f_{j}(E))>0,
\label{c}
\end{equation}
where $d_{\inf }(f_{i}(E),f_{j}(E))$ is the distance between $f_{i}(E)$\ and
$f_{j}(E)$. The \emph{similarity ratio} of $f_{i}\in \Psi $ is denoted by $%
r_{i}\in (0,1)$, and we write
\begin{equation}
r_{\min }:=\min_{i=1,\ldots ,m}r_{i\text{ }}\qquad \text{and \qquad }r_{\max
}:=\max_{i=1,\ldots ,m}r_{i}\text{.}  \label{rminmax}
\end{equation}

Both algorithms are based on the self-similar tiling principle stated in
\cite{[MO]} for self-similar sets satisfying the OSC. In \cite{[MO]} it was
shown that if $B$ is any closed subset (or tile) of a self-similar set $E$
such that $\mu (B)>0$, where $\mu $ is the invariant measure (see \eqref{mu}%
), then $E$ can be tiled, without any loss of $\mu $-measure, by a countable
collection of tiles that are images of $B$ under similitudes. Recall from
\cite{[HUTCH]} that, for self-similar sets satisfying the OSC, the measure $%
\mu $ is a multiple of any scaling measure, and in particular of the
packing, Hausdorff, or centered Hausdorff measure.

Taking an appropriate initial tile $B$ (one with minimal spherical density
in the case of the packing measure, and with maximal spherical density in
the case of the Hausdorff measure; see Remark \ref{rough idea}) we obtain
both an optimal packing or covering \cite{[LLM1]}, and the exact value of
the corresponding measure. Our method requires a particular form of the
separation condition, the SSC, in order to make the computation of the
metric measures feasible (see Section \ref{conclusions} for a discussion of
the computability of metric measures satisfying the OSC).

Metric measures suitable for studying the size of sets of Lebesgue measure
zero in $\mathbb{R}^{n}$, such as the Hausdorff, packing, and centered
Hausdorff measure ($H^{s}$, $P^{s}$, and $C^{s}$, respectively) have been
studied intensively in recent years. However, the challenging problem of
finding systematic methods for computing the values of these measures for a
general fractal set remains open. Much effort has been made in this
direction, and exact values and bounds for the measures of some fractal sets
are known already (see \cite{[Ay]}-\cite{[JZZ]}, \cite{[LLM1]}-\cite{[LLM4]}%
, \cite{[ME]}, \cite{[MO]}, \cite{[TTC]}, \cite{[Zh]}, and the references
therein).

In this direction, the algorithms presented in \cite{[LLM3]} and \cite%
{[LLM4]} can be seen as the first steps towards the systematic computation
of the centered Hausdorff and packing measures of a self-similar set. These
algorithms yield estimates of the corresponding measures for a wide class of
self-similar sets, taking as input the list of contracting similitudes
associated with the given set. It is important to note that in some cases,
such as for the class of Sierpinski gaskets with dimension less than or
equal to one, the packing measure algorithm has been useful not only for
estimating the value of $P^{s}$ on each particular set in the class, but
also for finding a formula for the packing measure of an arbitrary member of
the class. As shown in \cite[Theorem~2]{[LLM4]}, the information provided by
the algorithm can then be used to prove the formula. However, in many other
cases, such as for some plane self-similar sets of dimension greater than
one, the absence of the corresponding formula means that it is desirable to
know the accuracy of the numerical results obtained from the algorithms. In
\cite{[LLM3]} the centered Hausdorff measure algorithm was implemented for
some sets whose centered Hausdorff measures were available in the literature
and some other sets whose centered Hausdorff measures were still unknown. It
is remarkable that, in the first case, the optimal values were attained at
early iterations and, in the second case, the algorithm yielded conjectural
values that could be proved with the methods developed in \cite{[LLM4]}.
However, the rate of convergence of neither algorithm was known. This is the
problem we solve in this paper and the content of the next two main theorems.

\begin{theorem}
\label{main}Suppose that the system $\Psi =\{f_{1},\dots,f_{m}\}$ satisfies
the SSC. Then, for every $k\in \mathbb{N}^{+}$\textbf{\ }such that\textbf{\ }%
$c-3Rr_{\max }^{k}-2Rr_{\max }^{k+1}>0$\textbf{\ }and every $\tilde{M}_{k}$
as in \eqref{MK} there holds \textbf{\ }
\begin{equation}
|P^{s}(E) -\tilde{M}_{k}|\leq \varepsilon _{k}\text{,}  \label{gnrlbdd}
\end{equation}
where%
\begin{equation*}
\varepsilon _{k}:=\frac{s2^{s+1}RQ}{r_{\min }^{sq_{k}}}r_{\max }^{k},
\end{equation*}
$s=\dim _{H}(E)$, $q_{k}\in \mathbb{N}^{+}$ is such that $Rr_{\max
}^{q_{k}}\leq c-R2r_{\max }^{k+1}-2Rr_{\max }^{k}<Rr_{\max }^{q_{k}-1}$, and
\begin{equation}
Q:=\left\{
\begin{array}{c}
\left( \frac{c}{r_{\min }}\right) ^{s-1}\text{ if }s\geq 1, \\
\\
\left( c-2Rr_{\max }^{k}-2Rr_{\max }^{k+1}\right)^{s-1}\text{ if }s<1.%
\end{array}
\right.  \label{cte}
\end{equation}
\end{theorem}

\begin{theorem}
\label{mainc}Suppose that the system $\Psi =\{f_{1},\dots,f_{m}\}$ satisfies
the SSC. Then, for every $k\in \mathbb{N}^{+}$ and every $\tilde{m}_{k}$
given by \eqref{mk}, there holds
\begin{equation}
|C^{s}(E) -\tilde{m}_{k}|\leq \epsilon _{k}\text{.}  \label{gnrlbddc}
\end{equation}
where
\begin{equation*}
\epsilon _{k}:=\frac{s2^{s+1}R\mathcal{Q}}{r_{\min }^{qs}}r_{\max}^{k},
\end{equation*}
$s=\dim _{H}(E)$, $q\in \mathbb{N}^{+}$ is such that $Rr_{\max }^{q}\leq
c<Rr_{\max }^{q-1}$, and
\begin{equation}
\mathcal{Q}:=\left\{
\begin{array}{c}
R^{s-1}\text{ if }s\geq 1, \\
\\
c^{s-1}\text{ if }s<1.%
\end{array}
\right.  \label{ctec}
\end{equation}
\end{theorem}

Here, $\dim _{H}(E)$ and $R$ stand for the \emph{Hausdorff dimension} and
the \emph{diameter} of the self-similar set $E$.

As discussed in Sections \ref{packex} and \ref{HCex}, one of the most
important features of Theorems~\ref{main} and \ref{mainc} is that they
provide sharp bounds for the exact values of the corresponding measures.
Moreover, these bounds yield automatically estimates of the corresponding
measures, accurate in some cases to as many as $14$ decimal places. In the
difficult case of self-similar sets having dimensions greater than one, for
which less is known, we give examples with five decimal place accuracy. For
instance, applying Theorem~\ref{main} to the family of Sierpinski gaskets $%
\left\{ S_{r}\right\} $ with $\dim _{H}(S_{r})=-\frac{\log 3}{\log r}$ (see %
\eqref{sierclas} for a definition), yields $P^{s}(S_{0.37})\simeq 3.8728$
(see Table~\ref{times}) and $P^{s}(S_{0.42})\simeq 3.62$. We also get $%
P^{s}(K_{\frac{4}{10}})\simeq 5.27$, where $K_{\frac{4}{10}}$ is the plane
Cantor set of dimension $-\frac{\log 4}{\log 0.4}$. To our knowledge none of
these estimates were previously known.

However, the most important consequence of the combination of Theorems~\ref%
{main} and \ref{mainc} with the algorithms given in \cite{[LLM3]} and \cite%
{[LLM4]} is that it automatically provides an approximation to the value of
the measure of any self-similar set satisfying the SSC. We remark that the
precision of the results depends on the size of the contraction ratios.
Namely, the accuracy achieved improves as the contraction ratios decrease
(see \cite{[LLM3]}, \cite{[LLM4]} and Section~\ref{conclusions} below for a
detailed discussion). In particular, the examples given in Sections~\ref%
{packex} and \ref{HCex} show that the algorithms accurately recover the
known values of these measures for sets with dimension less than one.
Moreover, the results presented in this article serve to rule out certain
potential formulas for some classes of self-similar sets.

The paper is divided into two main sections, one devoted to the packing
measure and the other to the centered Hausdorff measure. In each case, we
first recall the relevant algorithm from \cite{[LLM3]} or \cite{[LLM4]},
although in the case of the centered Hausdorff measure we give some
improvements to the algorithm. This is done in Sections~\ref{recallpack}\
and \ref{recallch}. Then, we prove Theorems~\ref{main} and \ref{mainc} at
the ends of Sections \ref{ratepac} and \ref{ratehc}, respectively. It is
remarkable that these proofs do not use the convergence of the corresponding
algorithms, so the present note provides shorter alternative proofs of their
convergence. Finally, Sections \ref{packex} and \ref{HCex}\ are devoted to
analyzing the results obtained by applying Theorems~\ref{main} and \ref%
{mainc} to the examples given in \cite{[LLM3]} and \cite{[LLM4]}. These
numerical experiments have a twofold purpose. On the one hand they
illustrate the theoretical results, showing how the algorithms perform in
practice. On the other hand they offer quite complete information,
previously unavailable in the literature, on the exact values of the packing
and centered Hausdorff measures of the self-similar sets in three of the
most classic families of self-similar sets, namely, the central Cantor sets
in the line, the Sierpinski gaskets, and the Cantor sets in the unit square.
Finally, in Section \ref{conclusions} we discuss the computability of metric
measures on self-similar sets in view of the results obtained in this paper.

\begin{notation}
We denote the open and closed ball with center $x$ and radius $r$ by $%
B(x,d)=\{y\in \mathbb{R}^{n}:|x-y|<d\}$ and $\bar{B}(x,d)=\{y\in\mathbb{R}%
^{n}:|x-y|\leq d\}$, respectively. We use the notation $\partial B(x,d)$ for
the boundary of $B(x,d)$. Given $A\subset \mathbb{R}^{n}$, we write $|A|$
for the diameter of $A$ and $A^{c}=\{x\in \mathbb{R}^{n}:x\notin A\}$ for
the complement of $A$.

We write $s$ for the \emph{similarity dimension} of $E$, i.e., the unique
solution $s$ of $\sum_{i=1}^{m}r_{i}^{s}=1$. Sometimes we will refer to $s$
as the Hausdorff dimension, $\dim _{H}(E)$, of $E$, since the similarity and
Hausdorff dimension coincide when $E$ is a totally disconnected self-similar
set, as in the present note.

For the code space we use the following notation. Let $M:=\{1, \dots ,m\}$
and
\begin{equation*}
M^{k}=\{\mathbf{i}_{k}=(i_{1},\dots,i_{k}):i_{j}\in M\quad \forall
\,\,j=1,\dots,k\}.
\end{equation*}
Given $\mathbf{i}_{k}=i_{1}i_{2}\dots i_{k}\in M^{k}$, we write $f_{\mathbf{i%
}_{k}}$ for the similitude $f_{\mathbf{i}_{k}}=f_{i_{1}}\circ f_{i_{2}}\circ
\dots \circ f_{i_{k}}$ with similarity ratio $r_{\mathbf{i}%
_{k}}=r_{i_{1}}r_{i_{2}}\dots r_{i_{k}}$, and, given $A\subset \mathbb{R}%
^{n} $, we write $A_{\mathbf{i}_{k}}=f_{\mathbf{i}_{k}}(A)$, and refer to
the sets $E_{\mathbf{i}_{k}}=f_{\mathbf{i}_{k}}(E)$ as the \emph{cylinder
sets of generation }$k$. In particular, the sets $E_{i}=f_{i}(E)$, $i\in M$,
are called \emph{basic cylinder sets}.

We denote by $\mu $ the \emph{natural probability measure}, or \emph{%
normalized Hausdorff measure}, defined on the ring of cylinder sets by
\begin{equation}
\mu (E_{\mathbf{i}}) =r_{\mathbf{i}}^{s},\qquad \forall\,\, \mathbf{i\in
\cup }_{k=1}^{\infty }M^{k},  \label{mu}
\end{equation}%
and then extended to Borel subsets of $E$ (see \cite{[HUTCH]}).
\end{notation}

\begin{remark}
\label{rough idea} With the above notation, the idea underlying the
estimation of $P^{s}(E)$ and $C^{s}(E)$ can be summarized as follows: Find
the minimum and the maximum of the spherical densities $\frac{\mu (B(x,r))}{%
(2r)^{s}}$ on suitable families of balls. The inverse of the minimum is the
desired estimate for $P^{s}(E)$ and the inverse of the maximum is that for $%
C^{s}(E)$. Furthermore, these estimates give valuable additional information
about the behavior of $\mu $. For, if we let
\begin{equation*}
Spec=\left\{ \lim_{k\rightarrow \infty }\frac{\mu (B(x,r_{k}))}{(2r_{k})^{s}}%
:x\in E\text{ and }\lim_{k\rightarrow \infty }r_{k}=0\right\}
\end{equation*}
be the full range of limiting values of the spherical densities of $\mu $ on
balls, then the\ interval $[P^{s}(E)^{-1},C^{s}(E)^{-1}]$ is the minimal
interval that contains $Spec$ or spectral range of the density of $\mu$ (see
\cite{[LLM2]} and \cite{[MO]}).
\end{remark}

\section{Packing measure}

The \emph{packing measure} of a compact set $A$ with finite packing
premeasure can be defined by
\begin{equation*}
P^{s}(A) =\lim_{\delta \rightarrow 0}P_{\delta }^{s}(A),
\end{equation*}
where
\begin{equation*}
P_{\delta }^{s}(A) =\sup \left\{ \sum_{i=1}^{\infty }\left\vert
B_{i}\right\vert ^{s}:\left\vert B_{i}\right\vert \leq \delta ,\, i=1,2,3,
\dots\right\}
\end{equation*}
is a set function nondecreasing with respect to $\delta$ and the supremum is
taken over all countable collections of disjoint Euclidean balls centered in
$A$ and having diameters smaller than $\delta$ (see \cite{[Fen]}). Recall
that, as explained in \cite{[S]}-\cite{[T]}, a two-stage definition is
needed for general Euclidean sets.

In the specific case of self-similar sets much effort has been made to find
a simplified definition of $P^{s}$ suitable for computation. In \cite{[MO]}
it was shown how the above one-stage definition allows a characterization of
$P^{s}$ in terms of density functions which later facilitated tackling the
computability problem algorithmically (see \cite{[LLM1]} and \cite{[TC]}).
Next, we see that this characterization is also central for proving the rate
of convergence of the packing measure algorithm given by Theorem~\ref{main}.

\subsection{Previous results: Packing measure algorithm\label{recallpack}}

We begin with the following formula for the packing measure used in \cite%
{[LLM4]} as a starting point for the construction of an efficient algorithm.
For a self-similar set $E$ satisfying the SSC,
\begin{equation}
P^{s}(E) =\max \left\{ h(x,d):(x,d) \in E\times \lbrack c,\frac{c}{r_{\min }}%
]\right\} \text{,}  \label{packssc3}
\end{equation}
where $h(x,d):=\frac{(2d)^{s}}{\mu (B(x,d))}$. In \cite[Theorem~1]{[LLM4]},
there was proved the more general characterization of $P^{s}(E)$ as
\begin{eqnarray}
P^{s}(E) &=&\max \left\{ h(x,d):x\in E,d\leq a\right\}  \label{packssc} \\
&=&\max \left\{ h(x,d):x\in E,ar_{\min }\leq d\leq a\right\} ,\text{ }
\label{packssc2}
\end{eqnarray}
\bigskip where $a\in (0,\frac{c}{r_{\min }}]$ (see \eqref{SSC}, \eqref{c},
and \eqref{rminmax}, for the meaning of the notation used here). The reason
for choosing $a=\frac{c}{r_{\min }}$ in \eqref{packssc3} is to increase the
efficiency of the algorithm by both reducing the cardinality of the set of
balls on which the maximum is to be computed and increasing the radii of
these balls. However, for convenience, in the proof of Theorem~\ref{main} we
shall use \eqref{packssc2} in the form
\begin{equation}
P^{s}(E) =\max \left\{ h(x,d) :x\in E,b\leq d\leq \frac{b}{r_{\min }}%
\right\} ,  \label{corol}
\end{equation}
where $b\in (0,c]$.

Next, we recall the algorithm developed in \cite{[LLM4]} for computing the
value of $P^{s}(E)$ via approximations of the maximal value of $h(x,d)$.

\begin{algorithm}[Packing measure algorithm]
\label{case1}

\textbf{Input of the Algorithm}: The system, $\Psi$, of contracting
similitudes and $k_{\max }$, the iteration on which the algorithm's run was
stopped.

Let $k \le k_{\max }$ such that $c-2Rr_{\max}^{k}-2Rr_{\max }^{k+1}>0$.

\begin{enumerate}
\item \textbf{Construction of }$\mathbf{A}_{k}$. Let\ $A_{1}=\cup _{i\in
M}\{x\in \mathbb{R}^{n}:f_{i}(x)=x\}$ be the set consisting of the $m$ fixed
points of the similitudes in $\Psi $. For every $k\in \mathbb{N}^{+}$, let $%
A_{k}=S\Psi (A_{k-1})$ be the set formed by the $m^{k}$ points obtained by
applying $S\Psi (x)=\bigcup\limits_{i\in M}f_{i}(x)$ to each of the $m^{k-1}$
points of $A_{k-1}$.

\begin{notation}
\label{not2} For every $x\in A_{k}$ we denote by $\mathbf{i}%
_{k}^{x}=i_{1}^{x}\dots i_{k}^{x}\in M^{k}$ the unique sequence of length $k$
such that $x=f_{\mathbf{i}_{k}^{x}}(y)$ for some $y\in A_{1}$. Then $E_{%
\mathbf{i}_{k}^{x}}=f_{\mathbf{i}_{k}^{x}}(E)$ denotes the unique cylinder
set of generation $k$ such that $x\in E_{\mathbf{i}_{k}^{x}}$. Observe that
\begin{equation*}
x\in A_{k}\setminus A_{k-1}\iff i_{k}^{x}\neq i_{k-1}^{x}.
\end{equation*}
\end{notation}

\item \textbf{Generation of the list of distances.}

This step consists in computing the set
\begin{equation}
\Delta _{k}:=\{dist(x,y):(x,y) \in A_{k}\times A_{k}\}  \label{dist}
\end{equation}
of distances between pairs of points in $A_{k}$. It is important to note
that $\Delta _{k}\subset \Delta _{k+1}$ since, by construction, $%
A_{k}\subset A_{k+1}$. Hence, the computation of the set $\Delta _{k-1}$ of
distances should be avoided in the construction of $\Delta _{k}$. Therefore,
we shall calculate only those distances $dist(x,y) \in \Delta _{k}$ where
\begin{equation*}
i_{k}^{x}\neq i_{k-1}^{x}\text{ or }i_{k}^{y}\neq i_{k-1}^{y}\text{.}
\end{equation*}
For every $k\in \mathbb{N}^{+}$ we write $\Delta _{k}^{0}:=\{dist(x,y) \in
\Delta _{k}:i_{k}^{x}\neq i_{k-1}^{x}$ or $i_{k}^{y}\neq i_{k-1}^{y}\}$ with
$\Delta _{1}^{0}=\Delta_{1}$ and we write $\Delta _{k}=\Delta _{k}^{0}\cup
\Delta _{k-1}$.

Henceforth, we assign the code $(\mathbf{i}_{k}^{x}$, $\mathbf{i}_{k}^{y})$
to each $dist(x,y) \in \Delta _{k}$ and refer to $(\mathbf{i}_{k}^{x}$, $%
\mathbf{i}_{k}^{y})$ as the $k$-\emph{address} of $dist(x,y)$. Observe that
the $(k+1)$-\emph{address} of $dist(x,y) \in \Delta _{k}$ is\
\begin{equation*}
(\mathbf{i}_{k}^{x}i_{k}^{x},\mathbf{i}_{k}^{y}i_{k}^{y})=(i_{1}^{x}\dots
i_{k}^{x}i_{k}^{x},i_{1}^{y}\dots i_{k}^{y}i_{k}^{y}).
\end{equation*}

\item \textbf{Construction of }$\mathbf{\mu }_{k}$. Given $k\in \mathbb{N}%
^{+}$, set
\begin{equation}
\mu _{k}(x) =r_{\mathbf{i}_{k}^{x}}^{s}\,\, \forall \,\,x\in A_{k}.
\label{muigulapto}
\end{equation}
Then,
\begin{equation}  \label{muk}
\mu _{k}=\sum_{x\in A_{k}}r_{\mathbf{i}_{k}^{x}}^{s}\delta _{x}
\end{equation}
is a discrete probability measure supported on the $m^{k}$ points of $A_{k}$.

\item \textbf{Construction of }$\tilde{M}_{k}$.

Given $x\in A_{k}$:

\begin{itemize}
\item[4.1] Rank in increasing order the distances $d\in \Delta _{k}$
containing the letter $\mathbf{i}_{k}^{x}$ in their addresses and such that $%
d\leq \frac{c}{r_{\min }}$.


\item[4.2] Let $0 = d_{1}^{x}\leq d_{2}^{x}\leq \dots \leq d_{m_{x}}^{x} $
be the list of ordered distances, where $m_{x}\in \{1,\dots ,m^{k}\}$. For
every $j\in \{1,\dots ,m_{x}\}$, let $t_{j}\leq j$ be such that $%
d_{j}^{x}=d_{j-1}^{x}= \dots =$ $d_{t_{j}}^{x}\neq d_{t_{j}-1}^{x}$. Then

\begin{equation}
\mu _{k}(B(x,d_{j}^{x})) :=\sum_{q=1}^{t_{j}-1}r_{\mathbf{i}%
_{k}^{x_{q}}}^{s},  \label{mukigualball}
\end{equation}
where $x_{q}\in A_{k}$ is the point chosen for calculating the distance $%
d_{q}^{x}=dist(x,x_{q})$, for every $q=1,\dots ,j$ and $\mu
_{k}(B(x,d_{1}^{x}))=0$.

In the particular case when $r_{i}=r$ for all $i\in M$, \eqref{mukigualball}
simplifies to
\begin{equation*}
\mu _{k}(B(x,d_{j}^{x}))=\frac{t_{j}-1}{m^{k}}.
\end{equation*}

Compute
\begin{equation}
h_{k}(x,d_{j}^{x}) :=\frac{(2d_{j}^{x})^{s}}{\mu _{k}(B(x,d_{j}^{x}))}=\frac{%
(2d_{j}^{x})^{s}}{\sum_{q=1}^{t_{j}-1}r_{\mathbf{i}_{k}^{x_{q}}}^{s}}
\label{fkigual}
\end{equation}
only for those distances $d_{j}^{x}$ in the list satisfying
\begin{equation*}
0<c-2Rr_{\max }^{k}-2Rr_{\max }^{k+1}\leq d_{j}^{x}.
\end{equation*}

\item[4.3] Find the maximum
\begin{equation*}
M_{k}(x) :=\max \{h_{k}(x,d_{j}^{x}) :c-2Rr_{\max }^{k}-2Rr_{\max
}^{k+1}\leq d_{j}^{x}\leq \frac{c}{r_{\min }}\}
\end{equation*}
of the values computed in 4.2.

\item[4.4] Repeat steps 4.1-4.3 for each $x\in A_{k}$.

\item[4.5] Take the maximum
\begin{equation}
\tilde{M}_{k}:=\max \{M_{k}(x) :x\in A_{k}\}  \label{MK}
\end{equation}
of the $m^{k}$ values computed in step 4.4.
\end{itemize}
\end{enumerate}
\end{algorithm}

Observe that, for some $(\tilde{x}_{k},\tilde{y}_{k}) \in A_{k}\times A_{k}$,

\begin{align}  \label{mktildedef}
\begin{split}
\tilde{M}_{k} &:=h_{k}(\tilde{x}_{k},dist(\tilde{x}_{k},\tilde{y}_{k})) \\
&=\max \{h_{k}(x,dist(x,y) ) :(x,y) \in A_{k}\times A_{k}\text{ and }
c-2Rr_{\max }^{k}-2Rr_{\max }^{k+1}\leq dist(x,y) \leq \frac{c}{r_{\min }}\}.
\end{split}%
\end{align}

\begin{notation}
\label{notdist}In what follows we use the following notation. We denote by $%
D_{k}^{x}$ the set of distances selected in steps 4.1 and 4.2, and we let $%
D_{k}:=\cup _{x\in A_{k}}D_{k}^{x}$. We refer to $D_{k}^{x}$ as the set of
admissible distances for $x\in A_{k} $. Note that
\begin{equation}
D_{k}^{x}\subset D_{k}\subset \lbrack c-2Rr_{\max}^{k}-2Rr_{\max}^{k+1},
\frac{c}{r_{\min }}].  \label{adm}
\end{equation}
\end{notation}

It is important to note that the balls admissible in the algorithm have
radii in the interval $[c-2Rr_{\max}^{k}-2Rr_{\max}^{k+1},\frac{c}{r_{\min }}%
]\supset \lbrack c,\frac{c}{r_{\min }}]$ (see \eqref{packssc3}). This
containment helps in comparing the densities giving the packing measure with
those computed by the algorithm (see Section~\ref{ratepac}).

\subsection{Rate of convergence for the packing measure algorithm\label%
{ratepac}}

This section is devoted to proving Theorem~\ref{main}. One of the
difficulties one needs to overcome to show the rate of convergence %
\eqref{gnrlbdd} is to obtain a comparison between the measures $\mu $ and $%
\mu _{k}$ of a given ball (see \eqref{mu} and \eqref{muk}). Note that to
obtain a bound for $|P^{s}(E)-\tilde{M}_{k}|$ we need to compare the
densities $h(x,d)$ given in \eqref{corol} with those given in %
\eqref{mktildedef}. The following lemmas show that it is possible to
construct the approximating balls needed for such a comparison.

\begin{lemma}
\label{aproxball} For every $k\in \mathbb{N}^{+}$ and $(x,d)\in E\times
(0,\infty )$ such that $\partial B(x,d)\cap E\neq \emptyset$, there exists $%
(x^{\prime },d^{\prime }) \in A_{k}\times \lbrack d-2Rr_{\max
}^{k},d+2Rr_{\max }^{k}]$ such that

\begin{enumerate}
\item[(i)] $|x-x^{\prime }|\leq Rr_{\max }^{k}$,

\item[(ii)] $d^{\prime }=|x^{\prime }-y|$ for some $y\in A_{k}$,

\item[(iii)] $\mu (B(x,d))\geq \mu _{k}(B(x^{\prime },d^{\prime }))$.
\end{enumerate}
\end{lemma}

\begin{proof}
Let $k\in \mathbb{N}^{+}$ and $(x,d)\in E\times (0,\infty )$ be such that $%
\partial B(x,d)\cap E\neq \emptyset $. Take $x^{\prime}$ to be the unique
point in $A_{k}\cap E_{\mathbf{i}_{k}^{x}}$. Then (i) holds.

Set $L:=\{y\in A_{k}:E_{\mathbf{i}_{k}^{y}}\cap B^{c}(x,d)\neq \emptyset \}$
and let $d^{\prime }:=\min \{|x^{\prime }-y|:y\in L\}$. Observe that the
assumption $\partial B(x,d)\cap E\neq \emptyset $ implies that $L\neq
\emptyset $. Moreover, we can assume that $d^{\prime }\neq 0$ since,
otherwise, $x^{\prime }\in L$, $d\leq Rr_{\max }^{k}$, and $\mu
_{k}(B(x^{\prime },d^{\prime }))=0$ (see step 4.2 in Algorithm~\ref{case1})
and the lemma holds true.

In order to check the inequality $d^{\prime }\leq d+2Rr_{\max }^{k}$, take $%
z\in \partial B(x,d)\cap E$ and $y\in A_{k}\cap E_{\mathbf{i}_{k}^{z}}$.
Then, $y\in L$ and the triangle inequality together with (i) imply
\begin{equation*}
d^{\prime }\leq |x^{\prime }-y|\leq |x-x^{\prime }|+|x-z|+|y-z|\leq
d+2Rr_{\max }^{k}\text{.}
\end{equation*}
The inequality $d^{\prime }\geq d-2Rr_{\max }^{k}$ follows from (i) and the
triangle inequality by taking $y\in L$ such that $d^{\prime }=|x^{\prime}-y|
$ and $z\in E_{\mathbf{i}_{k}^{y}}\cap B^{c}(x,d)$, since then
\begin{equation*}
d\leq |x-z|\leq |x-x^{\prime }|+|x^{\prime }-y|+|y-z|\leq d^{\prime
}+2Rr_{\max }^{k}\text{.}
\end{equation*}
Finally, (iii) holds because for all $y\in B(x^{\prime },d^{\prime })\cap
A_{k}$ we have that $y\notin L$, whence $E_{\mathbf{i}_{k}^{y}}\subset
B(x,d) $. This, in turn, implies
\begin{eqnarray*}
\mu (B(x,d)) &=&\mu (\{E_{\mathbf{i}_{k}}:E_{\mathbf{i}_{k}}\subset
B(x,d)\})+\mu (\{E_{\mathbf{i}_{k}}\cap B(x,d):E_{\mathbf{i}_{k}}\cap
B^{c}(x,d)\neq \emptyset \}) \\
&\geq &\mu _{k}(\{y\in A_{k}:E_{\mathbf{i}_{k}^{y}}\subset B(x,d)\})\geq
\mu_{k}(B(x^{\prime },d^{\prime }))\text{,}
\end{eqnarray*}
which concludes the proof of (iii).
\end{proof}

\begin{lemma}
\label{aproxball2} For every $k\in \mathbb{N}^{+}$ and $(x,d)\in A_{k}\times
D_{k}^{x}$, there exists $d^{\prime }\in \lbrack d-Rr_{\max }^{k},d]$ such
that $\mu _{k}(B(x,d))\geq \mu (B(x,d^{\prime }))$.
\end{lemma}

\begin{proof}
Let $k\in \mathbb{N}^{+}$ and $(x,d)\in A_{k}\times D_{k}^{x}$ \textbf{(}see
Notational Convention \ref{notdist}\textbf{)}. Set $d^{\prime }:=\min
\{dist(x,E_{\mathbf{i}_{k}}):E_{\mathbf{i}_{k}}\cap B^{c}(x,d)\neq \emptyset
\}$. Note that, by definition of $D_{k}^{x}$, $\partial B(x,d)\cap A_{k}\neq
\emptyset $, and therefore $d^{\prime }\leq d$. In order to show $d^{\prime
}\geq d-Rr_{\max }^{k}$, choose $y\in E$ such that $d^{\prime }=|x-y|$ and $%
z\in B^{c}(x,d)\cap E_{\mathbf{i}_{k}^{y}}$. Then, by the triangle
inequality,
\begin{equation*}
d\leq |x-z|\leq |x-y|+|y-z|\leq d^{\prime }+Rr_{\max }^{k}.
\end{equation*}
Finally, the inequality $\mu _{k}(B(x,d))\geq \mu (B(x,d^{\prime }))$
follows because $B(x,d^{\prime })\cap E\subset \{y\in E:E_{\mathbf{i}%
_{k}^{y}}\subset B(x,d)\}$ and, therefore
\begin{eqnarray*}
\mu (B(x,d^{\prime })) &\leq &\mu (\{y\in E:E_{\mathbf{i}_{k}^{y}}\subset
B(x,d)\}) \\
&=&\mu _{k}(\{y\in A_{k}:E_{\mathbf{i}_{k}^{y}}\subset B(x,d)\})\leq
\mu_{k}(B(x,d)).
\end{eqnarray*}
\end{proof}

\begin{remark}
\noindent

\begin{enumerate}
\item Observe that, if $q_{k}\in \mathbb{N}^{+}$\ is as in Theorem~\ref{main}%
, then $E_{\mathbf{i}_{q_{k}}^{x}}\subset B(x,d)$\ for any $(x,d)\in
A_{k}\times D_{k}^{x}$\ and, therefore,
\begin{equation}
\mu _{k}(B(x,d))\geq \mu (E_{\mathbf{i}_{q_{k}}^{x}})\geq r_{\min }^{q_{k}s}.
\label{bd0}
\end{equation}

\item In the proof of Theorem~\ref{main} we shall use the following result
from \cite{[LLM4]}: Given $a\in (0,\frac{c}{r_{\min }}]$ and $%
(x_{0},d_{0})\in E\times \lbrack ar_{\min },a]$ such that $%
P^{s}(E)=h(x_{0},d_{0})=\frac{(2d_{0})^{s}}{\mu (B(x_{0},d_{0}))}$, then
\begin{equation}
\partial B(x_{0},d_{0})\cap E\neq \emptyset \text{.}  \label{nonemptybdd}
\end{equation}
\end{enumerate}
\end{remark}

Now we are ready to prove our main result.

\begin{proof}[Proof of Theorem~\protect\ref{main}]
We divide the proof into two cases: $P^{s}(E)\geq \tilde{M}_{k}$ and $%
P^{s}(E)\leq \tilde{M}_{k}$.

Let $k\in \mathbb{N}^{+}$. Suppose first that $P^{s}(E)\geq \tilde{M}_{k}$
and let $\mathcal{B}:= \{(x,d): x\in E \text{ and } d \in \lbrack
c-2Rr_{\max}^{k+1},\frac{c-2Rr_{\max}^{k+1}}{r_{\min }} ]\}$. Take $(\tilde{x%
},\tilde{d})\in \mathcal{B}$ such that
\begin{equation}
h(\tilde{x},\tilde{d})=P^{s}(E)=\max \left\{ h(x,d):(x,d)\in \mathcal{B}
\right\}  \label{packelec}
\end{equation}
(see \eqref{packssc2}). By \eqref{nonemptybdd} we know that $\partial B(%
\tilde{x},\tilde{d})\cap E\neq \emptyset $, so we can apply Lemma~\ref%
{aproxball} with $(x,d)=(\tilde{x},\tilde{d})$ and take $(x^{\prime},d^{%
\prime })\in A_{k}\times \lbrack \tilde{d}-2Rr_{\max }^{k},\tilde{d}%
+2Rr_{\max }^{k}]\subset A_{k}\times \lbrack c-2Rr_{\max
}^{k}-2Rr_{\max}^{k+1},\frac{c}{r_{\min }}]$. It is then clear, by %
\eqref{adm} that $0<d^{\prime }\in D_{k}^{x^{\prime }}$, and hence $%
B(x^{\prime },d^{\prime })$ is an admissible ball for the algorithm. This,
together with \eqref{mktildedef}, \eqref{bd0}, \eqref{packelec}, (iii) of
Lemma~\ref{aproxball}, and the mean value theorem, gives
\begin{eqnarray*}
P^{s}(E)-\tilde{M}_{k} &\leq &\frac{(2\tilde{d})^{s}}{\mu (B(\tilde{x},%
\tilde{d}))}-\frac{(2d^{\prime })^{s}}{\mu _{k}(B(x^{\prime },d^{\prime }))}
\\
&\leq &2^{s}\frac{(\tilde{d})^{s}-(\tilde{d}-2Rr_{\max}^{k})^{s}}{\mu
_{k}(B(x^{\prime },d^{\prime }))}\leq \frac{s2^{s+1}RQ}{r_{\min }^{sq_{k}}}%
r_{\max }^{k},
\end{eqnarray*}
where $Q$ is as in \eqref{cte}.

Now, suppose that $P^{s}(E)\leq \tilde{M}_{k}$ and let $(\tilde{x},\tilde{d}%
)\in A_{k}\times D_{k}^{\tilde{x}}$ be such that $\tilde{M}_{k}=h_{k}(\tilde{%
x},\tilde{d})$. Take $d^{\prime }$ as in Lemma~\ref{aproxball2} with $(x,d)=(%
\tilde{x},\tilde{d})$. Then $\mu _{k}(B(\tilde{x},\tilde{d}))\geq \mu (B(%
\tilde{x},d^{\prime }))$ with $(\tilde{x},d^{\prime })\in A_{k}\times
\lbrack \tilde{d}-Rr_{\max }^{k},\tilde{d}]\subset E\times \lbrack
c-3Rr_{\max }^{k}-2Rr_{\max }^{k+1},\frac{c}{r_{\min }}]$. Therefore, by (%
\ref{corol}), with $b=c-3Rr_{\max }^{k}-2Rr_{\max}^{k+1}>0$, we obtain
\begin{equation*}
h(\tilde{x},d^{\prime })\leq P^{s}(E)=\max \left\{ h(x,d):x\in E,d\in
\lbrack c-3Rr_{\max }^{k}-2Rr_{\max }^{k+1},\frac{c}{r_{\min }}]\right\} .
\end{equation*}
All this, together with the mean value theorem, gives

\begin{gather}
\tilde{M}_{k}-P^{s}(E)\leq \frac{(2\tilde{d})^{s}}{\mu _{k}(B(\tilde{x},
\tilde{d}))}-\frac{(2d^{\prime })^{s}}{\mu (B(\tilde{x},d^{\prime }))}
\notag \\
\leq 2^{s}\frac{\tilde{d}^{s}-(\tilde{d}-2Rr_{\max }^{k})^{s}}{\mu _{k}(B(
\tilde{x},\tilde{d}))}\leq \frac{s2^{s+1}RQ}{r_{\min}^{sq_{k}}}r_{\max }^{k},
\notag
\end{gather}
where $Q$ is as in \eqref{cte}. This concludes the proof of the theorem.
\end{proof}

\subsection{Examples\label{packex}}

Algorithm~\ref{case1} was tested in \cite{[LLM4]} with various different
classes of self-similar sets, including those previously studied in the
literature for which the value of the packing measure was known. Next, we
are going to analyze these same classes utilizing the point of view provided
by Theorem~\ref{main}. This allows, on the one hand, to obtain automatically
estimates for the results conjectured in \cite{[LLM4]} and, on the other
hand, to test the effectiveness of the algorithm when the exact value of the
packing measure is known. Moreover, we study some self-similar sets $E$ for
which, although the value of $P^{s}(E)$ is unknown, a conjecture can be
made. In these cases, the algorithm's output provides, for every $k\leq
k_{\max }$, an estimate $\mathbf{\tilde{M}}_{k}$ of the value of $P^{s}(E)$
and $100\%$ confidence intervals $I_{k}:=[\tilde{M}_{k}-\varepsilon _{k},
\tilde{M}_{k}+\varepsilon _{k}]$ (see Theorem~\ref{main}). This fact allows
us to reject the hypothesis $\alpha =P^{s}(E)$ when $\alpha \notin I_{k}$ .
If $\alpha \in I_{k}$, then the hypothesis cannot be ruled out as $%
\left\vert P^{s}(E)-\alpha \right\vert \leq 2\mathbf{\varepsilon }_{k}$ is
guaranteed. Here $k_{\max }$ denotes the iteration on which the algorithm's
run was stopped.

The results presented in Tables~\eqref{tabcatr}, \eqref{tabsier}, %
\eqref{times}, and \eqref{tabcant}, include, for completeness, the values of
the constants $s$, $q_{k_{\max }}$, $Q$, and $\mathbf{\varepsilon}_{k_{\max
}}$ involved in Theorem~\ref{main}. We note that, although all the
computations have been made using double precision arithmetic, the number of
decimal places displayed for the values of $Q$, $\mathbf{\varepsilon }%
_{k_{\max }}$, and $\mathbf{\tilde{M}}_{k_{\max }}$ has been reduced in
order to simplify the presentation

An interesting feature of the algorithm is that in some cases the output
stabilizes at an early iteration. The parameter $k_{stb}$ has been included
in the tables to indicate the iteration at which the algorithm output
stabilizes in the sense that, for all $k\in \lbrack k_{stb}, k_{\max }]\cap
\mathbb{N}^{+}$, there holds $\tilde{M}_{k}=\tilde{M}_{k_{stb}}$, after
rounding to $14 $ decimal places. \newline

The computer codes were written in Fortran 90. They were run on the HPC of
the Complutense University of Madrid (see
www.campusmoncloa.es/es/infraestructuras/eolo for technical description).

\begin{enumerate}
\item \textbf{Cantor sets in the real line}

Let $C_{r}$ be the linear Cantor set obtained as the attractor of the
iterated function system
\begin{equation}
\{f_{1}(x)=rx,\text{ \ }f_{2}(x)=1-r+rx\},\quad x\in \lbrack 0,1]\text{ and }
0<r<\frac{1}{2}.  \label{canreal}
\end{equation}

We know by \cite{[Fen0]} that, for all $r\in (0,\frac{1}{2})$,
\begin{equation}
P^{s}(C_{r})=\left( 2\frac{1-r}{r}\right) ^{s},  \label{cantconj}
\end{equation}
where $s=-\frac{\log 2}{\log r}$ is the similarity dimension of $C_{r}$.
Moreover, Algorithm~\ref{case1} was implemented in \cite{[LLM4]} for the
family $C_{r}$, yielding outputs that coincide with the corresponding values
given by \eqref{cantconj} (see $\tilde{M}_{20}$ in Table~\ref{tabcatr}). We
applied \eqref{gnrlbdd} to the class $C_{r}$ in order to check the
effectiveness of the bounds given by Theorem~\ref{main}. The results for the
final iteration $k_{\max }=20$ of the algorithm are presented in Table~\ref%
{tabcatr}.
\begin{table}[]
\begin{center}
\begin{tabular}{|c|c|c|c|c|c|c|}
\hline
$\mathbf{r}$ & $\mathbf{s}$ & $\mathbf{q_{20}}$ & $\mathbf{Q}$ & $\mathbf{%
\tilde{M}_{20}=P^{s}}\mathbf{({C}_{r})}$ & $\mathbf{\varepsilon_{20}}$ & $%
\mathbf{k_{stb}}$ \\ \hline
$1/4$ & $\frac{\log 2}{\log 4}$ & $1$ & $1.41421$ & $2.449489742783$ & $%
3.63798\times 10^{-12}$ & $2$ \\ \hline
$1/3$ & $\frac{\log 2}{\log 3}$ & $1$ & $1.50000$ & $2.398046289121$ & $%
1.68126\times 10^{-9}$ & $2$ \\ \hline
$0.38$ & $-\frac{\log 2}{\log 0.38}$ & $2$ & $1.49896$ & $2.333213028519$ & $%
5.56338\times 10^{-8}$ & $3$ \\ \hline
$0.383$ & $-\frac{\log 2}{\log 0.383}$ & $2$ & $1.49695$ & $2.327991242710$
& $6.58217\times 10^{-8}$ & $3$ \\ \hline
$0.45$ & $-\frac{\log 2}{\log 0.45}$ & $3$ & $1.35502$ & $2.172506324847$ & $%
3.98266\times 10^{-6}$ & $8$ \\ \hline
\end{tabular}%
\end{center}
\caption{Linear Cantor sets $C_{r}$. }
\label{tabcatr}
\end{table}

Observe that for these examples the accuracies of the values of $%
P^{s}(C_{r}) $ vary from ten to four decimal places (in the worse case):
\begin{eqnarray*}
\mathbf{2.4494897427}79 &\leq &P^{s}(C_{0.25})\leq \mathbf{2.4494897427}87,
\\
\mathbf{2.3980462}87 &\leq &P^{s}(C_{1/3})\leq \mathbf{2.3980462}91, \\
\mathbf{2.33321}29 &\leq &P^{s}(C_{0.38})\leq \mathbf{2.33321}31, \\
\mathbf{2.327991}17 &\leq &P^{s}(C_{0.383})\leq \mathbf{2.327991}31, \\
\mathbf{2.1725}02 &\leq &P^{s}(C_{0.45})\leq \mathbf{2.1725}11.
\end{eqnarray*}

\item \textbf{Sierpinski gaskets}

Let $S_{r}$ be the \emph{self-similar set} associated to the system $\Psi
=\{f_{1,}f_{2,}f_{3}\}$ where%
\begin{eqnarray}
f_{1}(\mathbf{x}) &=&r\mathbf{x},  \label{sierclas} \\
f_{2}(\mathbf{x}) &=&r\mathbf{x}+(1-r,0),  \notag \\
f_{3}(\mathbf{x}) &=&r\mathbf{x}+(1-r)\left( \frac{1}{2},\frac{\sqrt{3}}{2}
\right) ,  \notag
\end{eqnarray}
for $r\in (0,1)$ and $\mathbf{x}\in \mathcal{\mathbb{R}}^{2}$. If $r\in (0,%
\frac{1}{2})$, then $S_{r}$ is a Sierpinski gasket of similarity dimension $%
s=-\frac{\log 3}{\log r}$ satisfying the SSC.

By \cite{[LLM4]} we know that, for all $r\in (0,\frac{1}{3}]$,
\begin{equation}
P^{s}(S_{r})=g_{1}(r),\qquad \text{ where }g_{1}(r):=\left( 2\frac{1-r}{r}
\right) ^{s}\text{.}  \label{sierfor}
\end{equation}
The results presented in Table~\ref{tabsier} show that Theorem~\ref{main} in
combination with Algorithm~\ref{case1} provides quite complete information
on the packing measure of the family $\{S_{r}\}_{r\in (0,\frac{1}{2})}$:
When $r\in (0,\frac{1}{3}]$, Algorithm~\ref{case1} recovers the value of $%
P^{s}(S_{r})$, and when $r>\frac{1}{3}$, Theorem~\ref{main} provides an
approximate value for $P^{s}(S_{r})$.

In order to analyze the behavior of the family $S_{r}$ with respect to %
\eqref {sierfor}, Table~\ref{tabsier} is divided into three cases: $r\leq
1/3 $, where \eqref{sierfor} holds; $r>\frac{1}{3}$ and $g_{1}(r)\in
I_{k_{\max }}$ , where the supposition \eqref{sierfor} cannot be rejected as
$\left\vert P^{s}(E)-g_{1}(r)\right\vert \leq 2\mathbf{\varepsilon }%
_{k_{\max }}$ is guaranteed; and $r$ satisfying $g_{1}(r)\notin I_{k_{\max
}} $, where (\ref{sierfor}) can be ruled out. For completeness, Table~\ref%
{tabsier} includes the values of $g_{1}(r)$ as well.

\begin{table}[]
\begin{tabular}{|c|c|c|c|c|c|c|c|c|}
\hline
$\mathbf{r \leq \frac{1}{3}}$ & $\mathbf{s\leq 1}$ & $\mathbf{k_{\max }}$ & $%
\mathbf{q_{k_{\max }}}$ & $\mathbf{Q}$ & $\mathbf{\tilde{M}_{k_{\max
}}=P^{s}(K_{r})=g_{1}(r)}$ & $\mathbf{\varepsilon_{k_{\max }}}$ & $\mathbf{%
g_{1}(r)}$ & $\mathbf{k_{stb}}$ \\ \hline
\multicolumn{1}{|c|}{$1/27$} & \multicolumn{1}{|c|}{$\frac{\log 3}{\log 27}$}
& \multicolumn{1}{|c|}{$10$} & \multicolumn{1}{|c|}{$1$} &
\multicolumn{1}{|c|}{$1.05265$} & \multicolumn{1}{|c|}{$3.732511156817$} &
\multicolumn{1}{|c|}{$1.28830\times 10^{-14}$} & \multicolumn{1}{|c|}{$%
3.732511156817$} & $2$ \\ \hline
\multicolumn{1}{|c|}{$0.2$} & \multicolumn{1}{|c|}{$-\frac{\log 3}{\log 0.2}$%
} & \multicolumn{1}{|c|}{$12$} & \multicolumn{1}{|c|}{$1$} &
\multicolumn{1}{|c|}{$1.17602$} & \multicolumn{1}{|c|}{$4.134802967588$} &
\multicolumn{1}{|c|}{$3.16650\times 10^{-8}$} & \multicolumn{1}{|c|}{$%
4.134802967588$} & $2$ \\ \hline
\multicolumn{1}{|c|}{$1/4$} & \multicolumn{1}{|c|}{$\frac{\log 3}{\log 4}$}
& \multicolumn{1}{|c|}{$12$} & \multicolumn{1}{|c|}{$1$} &
\multicolumn{1}{|c|}{$1.15470$} & \multicolumn{1}{|c|}{$4.136854781603$} &
\multicolumn{1}{|c|}{$5.66827\times 10^{-7}$} & \multicolumn{1}{|c|}{$%
4.136854781603$} & $2$ \\ \hline
\multicolumn{1}{|c|}{$0.33$} & $-\frac{\log 3}{\log 0.33}$ & $12$ & $1$ & $%
1.00983$ & $4.009348546810$ & $1.99027\times 10^{-5}$ & \multicolumn{1}{|c|}{%
$4.009348546810$} & $3$ \\ \hline
\multicolumn{1}{|c|}{$1/3$} & \multicolumn{1}{|c|}{$1$} &
\multicolumn{1}{|c|}{$12$} & \multicolumn{1}{|c|}{$1$} &
\multicolumn{1}{|c|}{$1$} & \multicolumn{1}{|c|}{$4$} & \multicolumn{1}{|c|}{%
$2.25801\times 10^{-5}$} & $4$ & $2$ \\ \hline
$\mathbf{r>\frac{1}{3}}$ & \multicolumn{1}{|c|}{$\mathbf{s>1}$} &
\multicolumn{1}{|c|}{} & \multicolumn{1}{|c|}{} & \multicolumn{1}{|c|}{} &
\multicolumn{1}{|c|}{$\mathbf{\tilde{M}_{k_{\max }}},\ $ $\mathbf{%
g_{1}(r)\in I_{k_{\max }}}$} & \multicolumn{1}{|c|}{} &  &  \\ \hline
\multicolumn{1}{|c|}{$0.335$} & $-\frac{\log 3}{\log 0.335}$ & $12$ & $2$ & $%
0.99993$ & $3.995192673194$ & $7.24702\times 10^{-5}$ & $3.995192673194$ & $%
3 $ \\ \hline
\multicolumn{1}{|c|}{$0.36$} & \multicolumn{1}{|c|}{$-\frac{\log 3}{\log 0.36%
}$} & \multicolumn{1}{|c|}{$12$} & \multicolumn{1}{|c|}{$2$} &
\multicolumn{1}{|c|}{$0.98125$} & \multicolumn{1}{|c|}{$3.912076663518$} &
\multicolumn{1}{|c|}{$1.89640\times 10^{-4}$} & $3.912076663518$ & $3$ \\
\hline
\multicolumn{1}{|c|}{$0.365$} & $-\frac{\log 3}{\log 0.365}$ & $12$ & $2$ & $%
0.97322$ & $3.892897543783$ & $2.27289\times 10^{-4}$ & $3.892890309768$ & $%
3 $ \\ \hline
\multicolumn{1}{|c|}{$0.37$} & $-\frac{\log 3}{\log 0.37}$ & $12$ & $2$ & $%
0.96364$ & $3.872834140179$ & $2.71384\times 10^{-4}$ & $3.872817437454$ & $%
>k_{\max }$ \\ \hline
\multicolumn{1}{|c|}{$0.385$} & $-\frac{\log 3}{\log 0.385}$ & $12$ & $2$ & $%
0.92517$ & $3.807311991619$ & $4.51402\times 10^{-4}$ & $3.807142406190$ & $%
>k_{\max }$ \\ \hline
& \multicolumn{1}{|c|}{} & \multicolumn{1}{|c|}{} & \multicolumn{1}{|c|}{} &
\multicolumn{1}{|c|}{} & \multicolumn{1}{|c|}{$\mathbf{\tilde{M}_{k_{\max }}}%
, \ $ $\mathbf{g_{1}(r)\notin I_{k_{\max }}}$} & \multicolumn{1}{|c|}{} &  &
\\ \hline
\multicolumn{1}{|c|}{$0.39$} & \multicolumn{1}{|c|}{$-\frac{\log 3}{\log
0.39 }$} & \multicolumn{1}{|c|}{$13$} & \multicolumn{1}{|c|}{$2$} &
\multicolumn{1}{|c|}{$0.90895$} & \multicolumn{1}{|c|}{$3.783682419751$} &
\multicolumn{1}{|c|}{$2.06943\times 10^{-4}$} & $3.783386572225$ & $>k_{\max
}$ \\ \hline
\multicolumn{1}{|c|}{$0.42$} & \multicolumn{1}{|c|}{$-\frac{\log 3}{\log
0.42 }$} & \multicolumn{1}{|c|}{$15$} & \multicolumn{1}{|c|}{$3$} &
\multicolumn{1}{|c|}{$0.77328$} & \multicolumn{1}{|c|}{$3.629197993783$} &
\multicolumn{1}{|c|}{$2.83972\times 10^{-4}$} & $3.620358378152$ & $>k_{\max
}$ \\ \hline
\end{tabular}%
\caption{Sierpinski gaskets $\{S_{r}\}$}
\label{tabsier}
\end{table}

The packing measure bounds obtained from applying Theorem~\ref{main} to the
preceding examples are:\newline
\begin{eqnarray}
\mathbf{3.7325111568172}35 &\leq &P^{s}(S_{\frac{1}{27}})\leq \mathbf{%
3.7325111568172}62 ,  \notag \\
\mathbf{4.1348029}359 &\leq &P^{s}(S_{\frac{2}{10}})\leq \mathbf{4.1348029}
993,  \notag \\
\mathbf{4.13685}421 &\leq &P^{s}(S_{\frac{1}{4}})\leq \mathbf{4.13685}535,
\notag \\
\mathbf{4.0093}286 &\leq &P^{s}(S_{0.33})\leq \mathbf{4.0093}685,  \notag \\
3.9999774 &\leq &P^{s}(S_{\frac{1}{3}})\leq 4.0000226,  \notag \\
\mathbf{3.995}120 &\leq &P^{s}(S_{0.335})\leq \mathbf{3.995}266,  \notag \\
\mathbf{3.91}188 &\leq &P^{s}(S_{0.36})\leq \mathbf{3.91}227,  \notag \\
\mathbf{3.89}267 &\leq &P^{s}(S_{0.365})\leq \mathbf{3.89}313,  \notag \\
\mathbf{3.87}256 &\leq &P^{s}(S_{0.37})\leq \mathbf{3.87}311,  \notag \\
\mathbf{3.80}686 &\leq &P^{s}(S_{0.385})\leq \mathbf{3.80}777,  \notag \\
\mathbf{3.783}475 &\leq &P^{s}(S_{0.39})\leq \mathbf{3.783}890,
\label{S039bd} \\
\mathbf{3.62}891 &\leq &P^{s}(S_{0.42})\leq \mathbf{3.62}949.  \label{S042bd}
\end{eqnarray}
In view of Table~\ref{tabsier}, \eqref{sierfor} might also hold for $r$ in
some subinterval of $[\frac{1}{3},0.365)$. In particular, when the value of $%
r$ is one of $0.335$, $0.36$, $0.365$, $0.37$, and $0.385$, the algorithm
output approximates the value of $g_{1}(r)$ to $14$, $15$, $5$, $4$, and $3$
decimal place accuracy, respectively. However, this is not the case for
large $r$, as for $r$ equal to $0.39$ or $0.42$, we have that $%
g_{1}(r)\notin I_{k}$ (see \eqref{S039bd}, \eqref{S042bd} and the values of $%
g_{1}(0.39)$ and $g_{1}(0.42)$ in Table~\ref{tabsier}).

\begin{remark}
\label{timeq} Note that to improve the results for larger $r$, a larger
value of $k_{\max } $ would be required. However, it is necessary to
maintain an equilibrium between the gain in accuracy and the computational
time required (see Table~\ref{times}). The CPU times included in Table~\ref%
{times} are those that were necessary to obtain the values $\mathbf{\tilde{M}%
}_{k}$ for $k\leq k_{\max }$. Observe that the processing time needed for $%
k_{\max }=15$ is significantly bigger than that needed when $k_{\max }=13$
(see Section~\ref{general case packing} for further discussion).

\begin{table}[]
\begin{center}
\begin{tabular}{|c|c|c|c|}
\hline
$\mathbf{k_{\max }}$ & $\mathbf{\tilde{M}_{k_{\max }}}$ & $\mathbf{%
I_{k_{\max }}}$ & $\mathbf{CPU} $ \\ \hline
$11$ & $3.872849586344$ & $\left( \mathbf{3.87}211611,\mathbf{3.87}
358306\right) $ & $10$ minutes \\ \hline
$12$ & $3.872834140179$ & $\left( \mathbf{3.87}256275,\mathbf{3.87}
310553\right) $ & $92$ minutes \\ \hline
$13$ & $3.872826356688$ & $\left( \mathbf{3.872}72594,\mathbf{3.872}
92677\right) $ & $14$ hours and $42$ minutes \\ \hline
$14$ & $3.872821806279$ & $\left( \mathbf{3.872}78465,\mathbf{3.872}
85896\right) $ & $5$ days, $23$ hours, and $25$ minutes \\ \hline
$15$ & $3.872819763461$ & $(\mathbf{3.8728}0601,\mathbf{3.8728}3351)$ & $59$
days \\ \hline
\end{tabular}%
\end{center}
\caption{CPU times for $S_{0.37}$}
\label{times}
\end{table}
\end{remark}

Finally, Figure~\ref{figure} displays the values of $\tilde{M}%
_{k}-\varepsilon_{k}$, $\tilde{M}_{k}$, $\tilde{M}_{k}+\varepsilon_{k}$, and
$g_{1}(r)$ for $34$ equidistant values of $r$ in $[0.33,0.45]$. We have used
$k=12$ for $r\in \lbrack 0.33,0.445]$ and $k=13$ for $r\in \lbrack
0.45,0.495]$. The graphic shows the shape of the curve giving $\tilde{M}_{k}$
as a function of $r$, how the lengths of $I_{k}$ increase with $r$, and also
the differences between $\tilde{M}_{k}$ and $g_{1}(r)$ as functions of $r$.
It also proves that $g_{1}(r)$ is a lower bound for $P^{s}(S_{r})$. This
graph provides a computational alternative when the formula $g_{1}(r)$ is
not applicable.

\begin{figure}[]
\includegraphics[scale=0.6]{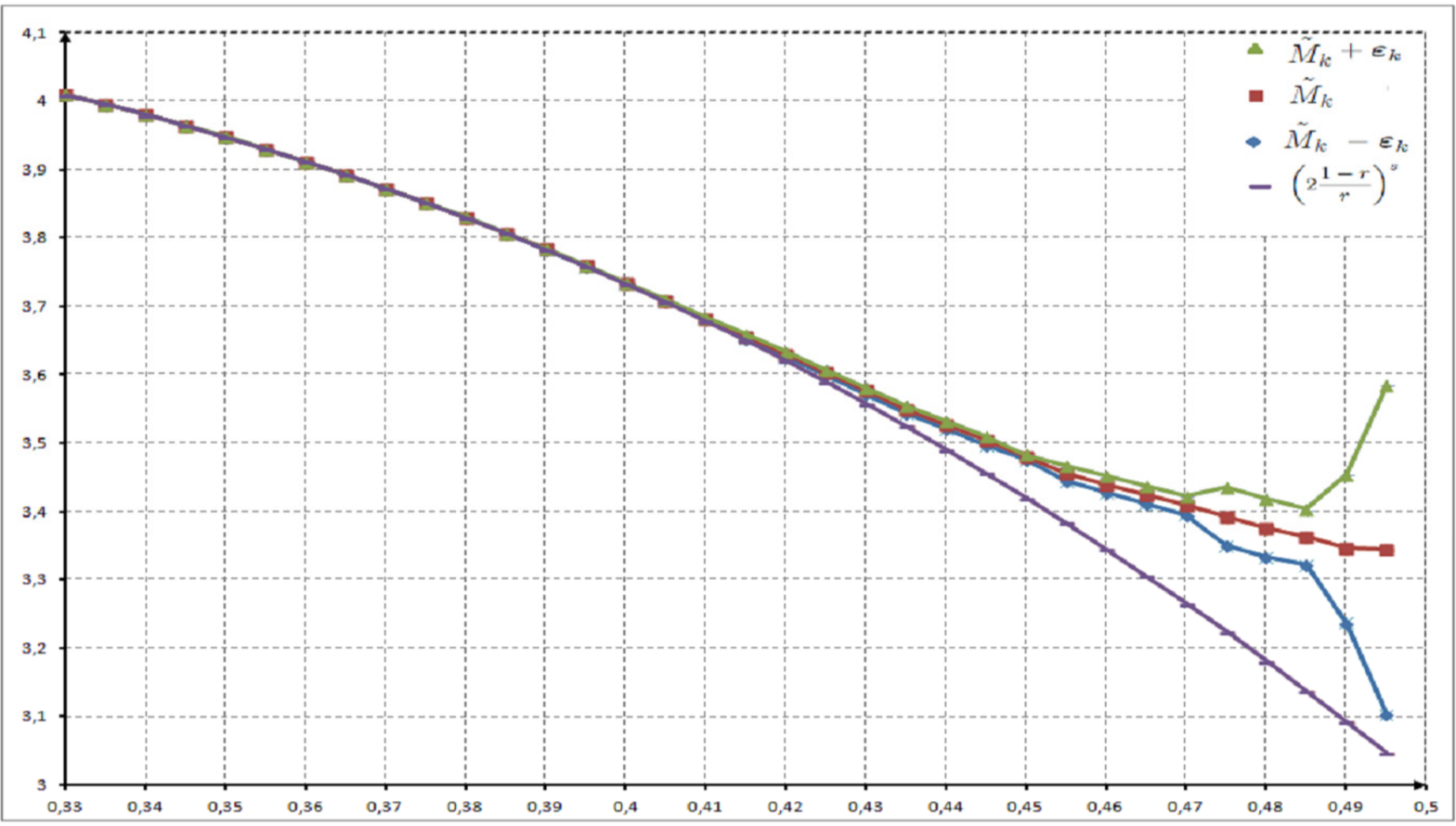}\newline
\caption{Values of $\tilde{M}_{k}-\protect\varepsilon_{k}$, $\tilde{M}_{k}$,
$\tilde{M}_{k}+\protect\varepsilon_{k}$, and $g_{1}(r)$ for $34$ equidistant
values of $r \in [0.33,0.45]$.}
\label{figure}
\end{figure}

\item \textbf{Planar Cantor sets}

Let $K_{r}$ be the attractor associated with the iterated function system $%
\Psi =\{f_{1},f_{2},f_{3},f_{4}\}$ where

\begin{align}
f_{i}(\mathbf{x}) =r\mathbf{x}+b_{i},\qquad i=1,2,3,4,\quad \mathbf{x}\in
\mathbb{R}^{2}\text{,\qquad }0<r<\frac{1}{2},  \label{planrcantor} \\
b_{1} =(0,0)\text{, }\ b_{2}=(1-r,0), \ \ b_{3}=(1-r,1-r), \ \ \text{and}\
b_{4}=(0,1-r).  \notag
\end{align}

Let $g_{2}(r):=\left( 2\frac{1-r}{r}\right)^{s}$ where $s=-\frac{\log 4}{%
\log r}$. In \cite{[LLM4]} it is proved that
\begin{equation}
P^{s}(K_{r})=g_{2}(r)  \label{cantorfor}
\end{equation}
for every $0<r\leq \frac{1}{4}$, and in \cite{[BZ]} the same formula is
shown to be true for $\frac{1}{4}<r<\frac{\sqrt{2}}{4}$.

As in the previous example, Table~\ref{tabcant} is divided into three cases
illustrating the behavior of the family $K_{r}$ with respect to %
\eqref{cantorfor}. When $0<r\leq \frac{\sqrt{2}}{4}$, in all cases the
output coincides at a very early iteration (see $k_{stb}$ in Table~\ref%
{tabcant}) with the corresponding value given by \eqref{cantorfor}. For $%
\frac{\sqrt{2}}{4}<r\leq 0.4$ we observe that $g_{2}(r)\in I_{k_{\max }}$
and thus, although \eqref{cantorfor} is proved only for $0<r\leq \frac{\sqrt{%
2}}{4}$ , this hypothesis cannot be discarded. In these cases we observe a
coincidence between the values given by $\mathbf{\tilde{M}}_{k}$ and $%
g_{2}(r)$ that varies from $12$ to $1$ decimal places. This is not the case
for $r=0.42$, as $g_{2}(r)\notin I_{k}$ and \eqref{cantorfor} can be ruled
out (see Table~\ref{tabcant} and \eqref{cant042}). We can now see the
advantage of combining Algorithm~\ref{case1} and Theorem \ref{main}, as we
are able to obtain an estimate $\mathbf{\tilde{M}}_{k}$ and a $100\%$
confidence interval $I_{k}$ for $P^{s}(K_{r})$ regardless of the existence
of an exact formula (see the estimates below). These examples also show that
$g_{2}(r)$ is a lower bound for $P^{s}(K_{r})$.
\begin{table}[tbp]
\begin{center}
\begin{tabular}{|c|c|c|c|c|c|c|c|c|}
\hline
$\mathbf{r\leq \frac{\sqrt{2}}{4}}$ & $\mathbf{s}$ & $\mathbf{k_{\max }}$ & $%
\mathbf{q_{k_{\max }}}$ & $\mathbf{Q}$ & $\mathbf{\tilde{M}%
_{k_{\max}}=P^{s}(K_{r})=g_{2}(r)}$ & $\mathbf{\varepsilon _{k_{\max }}}$ & $%
\mathbf{g_{2}(r)}$ & $\mathbf{k_{stb}}$ \\ \hline
$0.2$ & $-\frac{\log 4}{\log 0.2}$ & $10$ & $1$ & $1.07339$ & $%
5.996245070706 $ & $1.94598\times 10^{-6}$ & $5.996245070706$ & $2$ \\ \hline
$1/4$ & $1$ & $10$ & $1$ & $1$ & $6$ & $2.15792\times 10^{-5}$ & $6$ & $2$
\\ \hline
$0.3$ & $-\frac{\log 4}{\log 0.3}$ & $11$ & $2$ & $1.04453$ & $%
5.892731803791 $ & $2.14177\times 10^{-4}$ & $5.892731803791$ & $3$ \\ \hline
$0.35$ & $-\frac{\log 4}{\log 0.35}$ & $12$ & $2$ & $0.95179$ & $%
5.656172537869$ & $4.80038\times 10^{-4}$ & $5.656172537869$ & $4$ \\ \hline
$\mathbf{r>\frac{\sqrt{2}}{4}}$ &  &  &  &  & $\mathbf{\tilde{M}_{k_{\max }}}%
,$ $\mathbf{g_{2}(r)\in I_{k_{\max }}}$ &  &  &  \\ \hline
$0.36$ & $-\frac{\log 4}{\log 0.36}$ & $10$ & $2$ & $0.91421$ & $%
5.591584024577$ & $5.25726\times 10^{-3}$ & $5.591584024577$ & $4$ \\ \hline
$0.365$ & $-\frac{\log 4}{\log 0.365}$ & $10$ & $2$ & $0.89297$ & $%
5.557001901721$ & $6.05272\times 10^{-3}$ & $5.557001901721$ & $4$ \\ \hline
$0.37$ & $-\frac{\log 4}{\log 0.37}$ & $10$ & $2$ & $0.87012$ & $%
5.520873608633$ & $6.93986\times 10^{-3}$ & $5.520869632675$ & $4$ \\ \hline
$0.39$ & $-\frac{\log 4}{\log 0.39}$ & $11$ & $2$ & $0.76309$ & $%
5.361914850770$ & $4.47849\times 10^{-3}$ & $5.360487383353$ & $>k_{\max }$
\\ \hline
$0.395$ & $-\frac{\log 4}{\log 0.395}$ & $11$ & $3$ & $0.73263$ & $%
5.320007312123$ & $2.03399\times 10^{-2}$ & $5.316346629766$ & $>k_{\max }$
\\ \hline
$0.4$ & $-\frac{\log 4}{\log 0.4}$ & $12$ & $3$ & $0.70079$ & $%
5.277123200420 $ & $9.18961\times 10^{-3}$ & $5.270557940489$ & $>k_{\max }$
\\ \hline
$\mathbf{r}>\frac{\sqrt{2}}{4}$ &  &  &  &  & $\mathbf{\tilde{M}_{k_{\max }}}%
,$ $\mathbf{g_{2}(r)\notin I_{k_{\max }}}$ &  &  &  \\ \hline
$0.42$ & $-\frac{\log 4}{\log 0.42}$ & $11$ & $3$ & $0.56150$ & $%
5.12798012945 $ & $3.52941\times 10^{-2}$ & $5.070677295108$ & $>k_{\max }$
\\ \hline
\end{tabular}%
\end{center}
\caption{Planar Cantor sets $\{K_{r}\}$ }
\label{tabcant}
\end{table}

The bounds obtained from applying Theorem~\ref{main} to the preceding
examples are:

\begin{align}
\mathbf{5.99624}312& \leq P^{s}(K_{0.2})\leq \mathbf{5.99624}702,  \notag \\
5.99997842& \leq P^{s}(K_{\frac{1}{4}})\leq 6.00002158,  \notag \\
\mathbf{5.892}517& \leq P^{s}(K_{0.3})\leq \mathbf{5.892}946,  \notag \\
\mathbf{5.65}569& \leq P^{s}(K_{0.35})\leq \mathbf{5.65}666,  \notag \\
\mathbf{5.5}854& \leq P^{s}(K_{0.36})\leq \mathbf{5.5}968,  \notag \\
\mathbf{5.5}509& \leq P^{s}(K_{0.365})\leq \mathbf{5.5}631,  \notag \\
\mathbf{5.5}139& \leq P^{s}(K_{0.37})\leq \mathbf{5.5}279,  \notag \\
\mathbf{5.3}574& \leq P^{s}(K_{0.39})\leq \mathbf{5.3}664,  \notag \\
\mathbf{5.}2996& \leq P^{s}(K_{0.395})\leq \mathbf{5.}3404,  \notag \\
\mathbf{5.2}679& \leq P^{s}(K_{0.4})\leq \mathbf{5.2}864,  \notag \\
\mathbf{5}.0927& \leq P^{s}(K_{0.42})\leq \mathbf{5}.16326.  \label{cant042}
\end{align}
\end{enumerate}

\subsection{Computability of the packing measure: the general case\label%
{general case packing}}

A general pattern emerges from the above examples. For self-similar sets
with small contraction ratios there exists a formula that gives the exact
value of the packing measure, and the optimal density is attained for an
optimal ball, which can be found by the algorithm. As the contraction ratios
increase, these statements cease to be valid. This raises the problem of
whether, for a given case, there exists an optimal ball that can be computed
in finite time, and for which the exact value of $P^{s}(E)$ can be computed
to arbitrary accuracy. We say that a self-similar set with these properties
enjoys the finite time computability property.

Several things must happen for the finite time computability property to
hold. The optimal ball $B^{\ast }$ should be centered in one of the clouds $%
A_{k}$, and the boundary of the optimal ball should also lie in some $A_{k}$
. If these two conditions hold the optimal ball can be found in finite time,
but this does not guarantee that the exact value of $P^{s}(E)$ can be
computed in finite time, since the process estimating the exact value of $%
\mu(B^{\ast })$ can be infinite unless $B^{\ast }$ is a union of a finite
number of cylinders of the $k_{0}$th generation for some $k_{0}$, since then
$\mu (B^{\ast })=\mu _{k}(B^{\ast })$ for $k\geq k_{0}$ (Theorem 4.13 in
\cite{[LLM3]} illustrates this point). All these circumstances should be
considered as exceptional events, unless there were found some rigorous
proof that they must occur. Thus the \textit{general case} should be
considered as noncomputable in finite time.

If the finite time computability property holds in a particular case, a
theoretical argument based on geometric properties can give the exact
packing measure, but, in the general case, the only approach that can be
taken to calculate $P^{s}(E)$ is a computational one.

In order to illustrate the general case we present below Table~\ref{gener}.
It records the intermediate results of the algorithm for a unique
self-similar set, the Sierpinski gasket $S_{0.42}$.

\begin{table}[]
\begin{center}
\begin{tabular}{|c|c|c|c|c|c|}
\hline
$\mathbf{k}$ & $\mathbf{x}$ & $\mathbf{y}$ & $\mathbf{d}$ & $\mathbf{\tilde{M%
}}_{k}$ & $\mathbf{I}_{k}$ \\ \hline
\multicolumn{1}{|l|}{$5$} & \multicolumn{1}{|l|}{$(0,0)$} &
\multicolumn{1}{|l|}{$(0.25915848,0.02694808)$} & \multicolumn{1}{|l|}{$%
0.26055578$} & \multicolumn{1}{|l|}{$3.67050829$} & \multicolumn{1}{|l|}{$%
(2.00793066,\ 5.33308593)$} \\ \hline
\multicolumn{1}{|l|}{$6$} & \multicolumn{1}{|l|}{$(0,0)$} &
\multicolumn{1}{|l|}{$(0.13486912,\ 0.21096379)$} & \multicolumn{1}{|l|}{$%
0.25039050$} & \multicolumn{1}{|l|}{$3.65830695$} & \multicolumn{1}{|l|}{$%
(2.96002434,4.35658956)$} \\ \hline
\multicolumn{1}{|l|}{$7$} & \multicolumn{1}{|l|}{$(0,0)$} &
\multicolumn{1}{|l|}{$(0.24634452,\ 0.00475364)$} & \multicolumn{1}{|l|}{$%
0.24639038$} & \multicolumn{1}{|l|}{$3.64297340$} & \multicolumn{1}{|l|}{$%
(3.34969470,\ 3.93625210)$} \\ \hline
\multicolumn{1}{|l|}{$8$} & \multicolumn{1}{|l|}{$(0,0)$} &
\multicolumn{1}{|l|}{$(0.24590539,\ 0)$} & \multicolumn{1}{|l|}{$0.24590539$}
& \multicolumn{1}{|l|}{$3.63389479$} & \multicolumn{1}{|l|}{$(3.51071773,\
3.75707184)$} \\ \hline
\multicolumn{1}{|l|}{$9$} & \multicolumn{1}{|l|}{$(0,0)$} &
\multicolumn{1}{|l|}{$(0.24519182,\ 0.00275711)$} & \multicolumn{1}{|l|}{$%
0.24520732$} & \multicolumn{1}{|l|}{$3.63071511$} & \multicolumn{1}{|l|}{$%
(3.57898075,\ 3.68244948)$} \\ \hline
\multicolumn{1}{|l|}{$10$} & \multicolumn{1}{|l|}{$(0,0)$} &
\multicolumn{1}{|l|}{$(0.24673228,\ 0.00310930)$} & \multicolumn{1}{|l|}{$%
0.24675187$} & \multicolumn{1}{|l|}{$3.62998849$} & \multicolumn{1}{|l|}{$%
(3.60826005,\ 3.65171693)$} \\ \hline
\multicolumn{1}{|l|}{$11$} & \multicolumn{1}{|l|}{$(0,0)$} &
\multicolumn{1}{|l|}{$(0.24671071,\ 0.00411937)$} & \multicolumn{1}{|l|}{$%
0.24674510$} & \multicolumn{1}{|l|}{$3.62949853$} & \multicolumn{1}{|l|}{$%
(3.62037258,\ 3.63862448)$} \\ \hline
\multicolumn{1}{|l|}{$12$} & \multicolumn{1}{|l|}{$(0,0)$} &
\multicolumn{1}{|l|}{$(0.12700823,\ 0.21145014)$} & \multicolumn{1}{|l|}{$%
0.24666223$} & \multicolumn{1}{|l|}{$3.62928853$} & \multicolumn{1}{|l|}{$%
(3.62545563,\ 3.63312143)$} \\ \hline
\multicolumn{1}{|l|}{$13$} & \multicolumn{1}{|l|}{$(0,0)$} &
\multicolumn{1}{|l|}{$(0.24666388,\ 0.00281924)$} & \multicolumn{1}{|l|}{$%
0.24667999$} & \multicolumn{1}{|l|}{$3.62921523$} & \multicolumn{1}{|l|}{$%
(3.62760541,\ \ 3.63082505)$} \\ \hline
\multicolumn{1}{|l|}{$14$} & \multicolumn{1}{|l|}{$(0,0)$} &
\multicolumn{1}{|l|}{$(0.24663898,\ 0.00411937)$} & \multicolumn{1}{|l|}{$%
0.24667338$} & \multicolumn{1}{|l|}{$3.62921324$} & \multicolumn{1}{|l|}{$%
(3.62853711,\ \ 3.62988937)$} \\ \hline
\multicolumn{1}{|l|}{$15$} & \multicolumn{1}{|l|}{$(0,0)$} &
\multicolumn{1}{|l|}{$(0.24663671,\ 0.00424755)$} & \multicolumn{1}{|l|}{$%
0.24667328$} & \multicolumn{1}{|l|}{$3.62919799$} & \multicolumn{1}{|l|}{$%
(3.62891402,\ 3.62948197)$} \\ \hline
\end{tabular}%
\end{center}
\caption{Sierpinski gasket $S_{0.42}$}
\label{gener}
\end{table}

The columns in this table are: the number $k$ of iterations, the center $x$
and endpoint $y$ of the optimal ball at the $k$th iteration, the radius $d$,
the estimate $\tilde{M}_{k}$ for $P^{s}(S_{0.42})$ at the $k$th iteration,
and the interval $I_{k}$ to which we can be sure that $P^{s}(S_{0.42})$
belongs. For simplicity all the values are rounded to eight decimal places.
One can see that, in spite of the stabilization of $x$, the remaining values
change from iterate to iterate until the computational time is too big to
continue.

\begin{remark}
The case $r=0.42$ is a good example illustrating the frontiers of
computability of the packing measure. We needed $k_{\max }=15$ to obtain
only two digits of accuracy in the estimate of $P^{s}(S_{0.42})$. It is
difficult to increase $k_{\max }$ because the set $A_{15}$ consists of $%
3^{15}=14348907$ data points, and for this value the computation required
more than one month of CPU time.
\end{remark}

\section{Centered Hausdorff measure}

The centered Hausdorff measure is a variant of the Hausdorff measure. The
main difference between them is the nature of the coverings used in their
definitions. In the case of the centered Hausdorff measure the set of
coverings is restricted to closed balls centered at points in the given set
(see, e.g. \cite{[TC]}, for the standard definition and properties).
However, here, instead of the standard definition of $C^{s}$ we use the
following relation proved in \cite{[LLM1]} for totally disconnected
self-similar sets:
\begin{equation}
C^{s}(E)=\min \left\{ \bar{h}(x,d):x\in E\text{ \ and }c\leq d\leq R\right\}
\label{censsc}
\end{equation}
(see Section~\ref{intro} for the notational conventions). This can be
improved to

\begin{equation}
C^{s}(E)=\min \left\{ \bar{h}(x,d):(x,d)\in \mathcal{A}\right\} ,
\label{censsc2}
\end{equation}
with
\begin{equation}
\mathcal{A}=\left\{ (x,d)\in E\times \lbrack c,R]:\bar{B}(x,d)\cap E_{j}\neq
\emptyset \text{ for some }j\in M\text{ with }j\neq i_{1}^{x}\right\} .
\label{censsc3}
\end{equation}
This is so because, as argued in \cite[Remark 6]{[LLM1]}, any ball $B(x,d)$
with $B(x,d)\cap E\subset E_{i_{1}^{x}}$ can be enlarged to a ball $%
B(f_{i_{1}^{x}}^{-1}(x), \frac{d}{r_{i_{1}^{x}}})$ with $\bar{h}(x,d)=\bar{h}%
(f_{i_{1}^{x}}^{-1}(x), \frac{d}{r_{i_{1}^{x}}})$. The inequality $d<c$
ensures that $B(x,d)\cap E\subset E_{i_{1}^{x}}$ holds true, but, even if $%
d\geq c$, it might happen that $B(x,d)$ could still be enlarged. The
condition used in the above definition of $\mathcal{A}$ rules out, however,
any further enlargement of $B(x,d)$.

Similarly to the packing measure case, \eqref{censsc2} allows the
construction of an algorithm converging to the value\ of $C^{s}(E)$ through
an approximation of the minimal value of $\bar{h}(x,d):=\frac{(2d)^{s}}{\mu
( \bar{B}(x,d))}$. Observe that we are taking closed balls in the definition
of $\bar{h}(x,d)$ instead of the open ones used in the packing measure case.
Nevertheless, replacing open balls with closed balls in (\ref{censsc}) does
not make any difference in the limit as, by \cite{[MT]}, we know that $\mu
(\partial B(x,d))=0$. Moreover, using closed balls has proved to be
computationally more efficient.

\subsection{Previous results: centered Hausdorff measure algorithm\label%
{recallch}}

This section describes an improved version of the algorithm developed in
\cite{[LLM3]} for computing the centered Hausdorff measure of totally
disconnected self-similar sets. This new version has two novelties. On the
one hand, it allows a reduction of the number of calculations needed on each
step, making the algorithm faster at the expense of using more memory by
caching part of the calculations made in prior iterations instead of
recalculating them on each iteration. On the other hand, it uses the more
efficient condition \eqref{censsc2} in place of \eqref{censsc}.

As the structure of Algorithm~\ref{algocent} is very similar to that of the
algorithm for the packing measure, we start the description supposing that $%
A_{k}$, $\Delta _{k}$ and $\mu_{k}$ have already been constructed, so that
we can see the differences between the two algorithms.

\begin{algorithm}[Centered Hausdorff measure]
\label{algocent}\textbf{Input of the Algorithm}: The system of contracting
similitudes and the iteration $k_{\max}$ on which the algorithm's run is
stopped.

We begin the description of the algorithm with step 4 as the construction of
$A_{k}$, $\mu_{k}$, and the list of distances is the same as in the packing
measure case. Thus, assume steps 1, 2, and 3 are as in Algorithm~\ref{case1}%
, and let $k\in\mathbb{N}^+$ such that $k\le k_{\max}$.

\begin{enumerate}
\item[4] \textbf{Construction of }$\mathbf{\tilde{m}}_{k}$.

Given $x\in A_{k}$:

\begin{itemize}
\item[4.1] Rank in increasing order those distances $d\in \Delta _{k}$ that
contain the letter $\mathbf{i}_{k}^{x}$ in their addresses (see \eqref{dist}
for the notation).

\item[4.2] Let $0=d_{1}^{x}\leq d_{2}^{x}\leq \dots \leq d_{m^{k}}^{x}$ be
the list of ordered distances and, for every $j\in \{1,\dots ,m^{k}\}$, let $%
t_{j}\in \mathbb{N} $\ be such that $d_{j}^{x}=\dots=d_{j+t_{j}}^{x}\neq
d_{j+t_{j}+1}^{x}$. Then,
\begin{equation}
\mu_{k}(\bar{B}(x,d_{j}^{x})):=\sum_{q=1}^{j+t_{j}}r_{\mathbf{i}%
_{k}^{x_{q}}}^{s},  \label{mudifball}
\end{equation}
where $x_{q}\in A_{k}$ is the point chosen for calculating the distance
\textbf{\ }$d_{q}^{x}=d(x,x_{q})$, $q=1,\dots,j+t_{j}$.

Observe that, in the particular case when $r_{i}=r$ for all $i\in M$, we
have
\begin{equation*}
\mu _{k}(B(x,d_{j}^{x}))=\frac{j+t_{j}}{m^{k}}.
\end{equation*}

\item[4.3] Compute
\begin{equation}
\bar{h}_{k}(x,d_{j}^{x}):=\frac{(2d_{j}^{x})^{s}}{\mu_{k}(\bar{B}%
(x,d_{j}^{x}))}= \frac{(2d_{j}^{x})^{s}}{\sum_{q=1}^{j+t_{j}}r_{\mathbf{i}%
_{k}^{x_{q}}}^{s}}  \label{hbk}
\end{equation}
\textbf{only} for those distances $d_{j}^{x}$ in the list satisfying $%
j_{0}^{x}\leq j\leq m^{k}$, where
\begin{equation}
j_{0}^{x}=\min \{j\in \{1,\dots,m^{k}\}:d_{j}^{x}=|x-y|\text{ for some }y\in
A_{k}\text{ with }i_{1}^{x}\neq i_{1}^{y}\}\text{,}  \label{condition}
\end{equation}
according to \eqref{censsc2} and \eqref{censsc3}.

Henceforth we use the following notation. Given $k\in \mathbb{N}^{+}$ and $%
x\in A_{k}$, we define
\begin{eqnarray*}
\bar{D}_{k}^{x} :=\cup _{j=j_{0}^{x}}^{m^{k}}d_{j}^{x}
\end{eqnarray*}
and
\begin{eqnarray*}
\bar{D}_{k} :=\cup _{x\in A_{k}}\bar{D}_{k}^{x}.
\end{eqnarray*}

Observe that $\bar{D}_{k}$ takes values only within the interval $[c,R]$.

\item[4.4] Find the minimum
\begin{equation*}
\bar{m}_{k}(x)=\min \{\bar{h}_{k}(x,d_{j}^{x}):j=j_{0},\dots,m^{k}\}
\end{equation*}
of the values computed in step 4.3.
\end{itemize}

\item[5] Repeat step 4 for each $x\in A_{k}$.

\item[6] Take the minimum
\begin{equation}
\tilde{m}_{k}:=\min \{\bar{m}_{k}(x):x\in A_{k}\}  \label{mk}
\end{equation}
of the $m^{k}$ values computed in step 5. Note that
\begin{equation}
\tilde{m}_{k}:=\bar{h}_{k}(\tilde{x}_{k},\tilde{d}_{k}):= \frac{(2\tilde{d}%
_{k})^{s}}{\mu_{k}(\bar{B}(\tilde{x}_{k},\tilde{d}_{k}))}= \min_{x\in
A_{k}}\min_{d\in \bar{D}_{k}^{x}}\bar{h}_{k}(x,d).  \label{seal}
\end{equation}
\end{enumerate}
\end{algorithm}

\begin{remark}
\noindent

\begin{enumerate}
\item Note that, by \eqref{condition}, for any $x\in A_{k}$,
\begin{eqnarray}
d &\in &\bar{D}_{k}^{x}\iff d=|x-y|\text{ for some }y\in A_{k}\text{ and }
\label{condi} \\
&&\text{there exists }z\in A_{k}\cap \bar{B}(x,d)\text{ with }i_{1}^{z}\neq
i_{1}^{x}.  \notag
\end{eqnarray}

\item Let $q\in \mathbb{N}^{+}$ be such that $Rr_{\max }^{q}\leq c<Rr_{\max
}^{q-1}$. Then, for any $k\in \mathbb{N}^{+}$ and $(x,d)\in A_{k} \times
\bar{D}_{k}^{x}$, there holds $E_{\mathbf{i}_{q}^{x}}\subset \bar{B}(x,d)$,
whence
\begin{equation}
\mu _{k}(\bar{B}(x,d))\geq \mu (E_{\mathbf{i}_{q}^{x}})\geq r_{\min }^{qs}.
\label{bd0c}
\end{equation}
\end{enumerate}
\end{remark}

\subsection{Rate of convergence of the centered Hausdorff measure algorithm
\label{ratehc}}

This section is devoted to showing the rate of convergence of Algorithm~\ref%
{algocent}. Since the proof of Theorem~\ref{mainc} does not use the
convergence $C^{s}(E)=\lim_{k\rightarrow \infty}\tilde{m}_{k}$, this gives
an alternative proof of the convergence of Algorithm~\ref{algocent}.

As in Section~\ref{ratepac}, we show first that the construction of
appropriate approximating balls allows a comparison between the densities
given in \eqref{censsc2} with those computed by Algorithm~\ref{algocent}.
The proof of Theorem~\ref{mainc} is postponed to the end of the section.

\begin{lemma}
\label{aproxballc}Given $k\in\mathbb{N}^{+}$ and $(x,d)\in \mathcal{A}$
there exists $(x^{\prime },d^{\prime })\in A_{k}\times \bar{D}%
_{k}^{x^{\prime }}$ such that

\begin{enumerate}
\item[(i)] $|x-x^{\prime }|\leq Rr_{\max }^{k}$,

\item[(ii)] $d^{\prime }\leq d+2Rr_{\max }^{k}$,

\item[(iii)] $\mu (\bar{B}(x,d))\leq \mu _{k}(\bar{B}(x^{\prime },d^{\prime
}))$.
\end{enumerate}
\end{lemma}

\begin{proof}
Let $k\in \mathbb{N}^{+}$ and $(x,d)\in \mathcal{A}$ (see \eqref{censsc3}
for the notation). Take the unique point $x^{\prime }\in A_{k}$ such that $%
E_{\mathbf{i}_{k}^{x}}=E_{\mathbf{i}_{k}^{x^{\prime }}}$. Then (i) holds.
Now set $L:=\{y\in A_{k}:E_{\mathbf{i}_{k}^{y}}\cap \bar{B}(x,d)\neq
\emptyset \}$\ and $d^{\prime }:=\max \{|y-x^{\prime }|:y\in L\}$.

Observe that, by definition of $\mathcal{A}$, there exists $\mathbf{j}_k \in
M^k$ with $j_1\ne i_1^x$ and $\bar{B}(x,d) \cap E_{\mathbf{j}_{k}}\ne
\emptyset $. Moreover, taking $z \in A_k$ such that $E_{\mathbf{i}%
_{k}^{z}}=E_{\mathbf{j}_{k}}$, we obtain that $z\in L$ whence $%
|z-x^{\prime}|\le d^{\prime}$. This, in turn, implies that $d^{\prime} \in
\bar{D}_{k}^{x^{\prime }}$ because $z\in \bar{B}(x^{\prime},d^{\prime})\cap
A_k$ and $i_1^{x^{\prime}}=i_1^x \ne j_1=i_1^z$ (see \eqref{condi}).

The proof of (ii) follows from the triangle inequality, taking $t^{\prime
}\in L\cap \partial \bar{B}(x^{\prime },d^{\prime })$ and $t\in E_{\mathbf{i}%
_{k}^{t^{\prime }}}\cap \bar{B}(x,d)$:
\begin{equation*}
d^{\prime }=|x^{\prime }-t^{\prime }|\leq
|x^{\prime}-x|+|x-t|+|t-t^{\prime}|\leq d+2Rr_{\max }^{k}.
\end{equation*}

Finally, (iii) holds because $L\subset \bar{B}(x^{\prime},d^{\prime })\cap
A_{k}$ and hence
\begin{equation*}
\mu (\bar{B}(x,d))\leq \mu (\cup _{\mathbf{i}_{k}\in M^{k}}\{E_{_{\mathbf{i}%
_{k}}}:E_{_{\mathbf{i}_{k}}}\cap \bar{B}(x,d)\neq \emptyset \})=\mu
_{k}(L)\leq \mu_{k}(\bar{B}(x^{\prime },d^{\prime })).
\end{equation*}
\end{proof}

\begin{lemma}
\label{aproxballc2} Given $(x,d)\in A_{k}\times \bar{D}_{k}^{x}$, there
exists $d^{\prime }\in \lbrack c,R]$ with $(x,d^{\prime }) \in \mathcal{A},$
and such that

\begin{enumerate}
\item[(i)] $d^{\prime }\leq d+Rr_{\max }^{k}$ and

\item[(ii)] $\mu _{k}(\bar{B}(x,d))\leq \mu (\bar{B}(x,d^{\prime }))$.
\end{enumerate}
\end{lemma}

\begin{proof}
Let $P:=\{\mathbf{i}_{k}\in M^{k}:E_{\mathbf{i}_{k}}\cap \bar{B}(x,d)\neq
\emptyset \}$ and $L:=\bigcup_{\mathbf{i}_{k}\in P} E_{\mathbf{i}_{k}}$. Set
$d^{\prime }:=\max \{|y-x|:y\in L\}$. By definition $L\subset \bar{B}
(x,d^{\prime })\cap E$. Thus
\begin{equation*}
\mu _{k}(\bar{B}(x,d))\leq \mu _{k}(\{y\in A_{k}:E_{\mathbf{i}_{k}^{y}}\cap
\bar{B}(x,d)\neq \emptyset \})=\mu (L)\leq \mu (\bar{B}(x,d^{\prime }))
\end{equation*}
which proves (ii). The proof of (i) follows by taking $y\in L\cap \partial
\bar{B}(x,d^{\prime })$ and $z\in E_{\mathbf{i}_{k}^{y}}\cap \bar{B}(x,d)$
and applying the triangle inequality:
\begin{equation*}
d^{\prime }=|x-y|\leq |x-z|+|z-y|\leq d+Rr_{\max }^{k}.
\end{equation*}
Finally, $\eqref{condi}$ implies the existence of $z \in A_k \cap \bar{B}%
(x,d)$ with $i_{1}^{z} \ne i_1^x$. This proves $(x,d^{\prime}) \in \mathcal{A%
}$ since $E_{\mathbf{i}_k^z} \in L$ and $B(x,d^{\prime})\cap E_{\mathbf{i_k}%
^z} \ne \emptyset $
\end{proof}

We are now ready to prove our main result for the centered Hausdorff measure.

\begin{proof}[Proof of Theorem~\protect\ref{mainc}]
Suppose first that $\tilde{m}_{k}\geq C^{s}(E)$. Let $(x,d)\in \mathcal{A}$
be such that $C^{s}(E)=\frac{(2d)^{s}}{\mu (\bar{B}(x,d))}$ (see %
\eqref{censsc2}) and take $(x^{\prime},d^{\prime })\in A_{k}\times \bar{D}%
_{k}^{x^{\prime}}$ as in Lemma~\ref{aproxballc}. Then \eqref{seal}, (ii) and
(iii) of Lemma~\ref{aproxballc}, \eqref{bd0}, and the mean value theorem
imply
\begin{eqnarray*}
\tilde{m}_{k}-C^{s}(E) &\leq &\frac{(2d^{\prime })^{s}}{\mu _{k}(\bar{B}%
(x^{\prime },d^{\prime }))}-\frac{(2d)^{s}}{\mu (\bar{B}(x,d))} \\
&\leq &2^{s}\frac{(d^{\prime })^{s}-d^{s}}{\mu _{k}(\bar{B}(x^{\prime
},d^{\prime }))}\leq \frac{s2^{s+1}Q}{r_{\min }^{qs}}Rr_{\max }^{k},
\end{eqnarray*}
where $Q$ is as in \eqref{ctec}.

Finally, if $\tilde{m}_{k}\leq C^{s}(E)$, \eqref{seal}, the mean value
theorem, and Lemma~\ref{aproxballc2} with $(x,d)\in A_{k}\times \bar{D}%
_{k}^{x}$ such that $\tilde{m}_{k}=\frac{(2d)^{s}}{\mu _{k}(\bar{B}(x,d))}$,
imply
\begin{gather}
C^{s}(E)-\tilde{m}_{k}\leq \frac{(2d^{\prime })^{s}}{\mu (\bar{B}%
(x,d^{\prime }))}-\frac{(2d)^{s}}{\mu _{k}(\bar{B}(x,d))}  \notag \\
\leq 2^{s}\frac{(d^{\prime })^{s}-d^{s}}{\mu _{k}(\bar{B}(x,d^{\prime }))}%
\leq s2^{s}\mathcal{Q}R\frac{r_{\max }^{k}}{r_{\min }^{qs}},  \notag
\end{gather}
where $Q$ is as in \eqref{ctec} and $d^{\prime }$ is given by Lemma~\ref%
{aproxballc2}.
\end{proof}

\subsection{Examples\label{HCex}}

As in Section~\ref{packex}, we now analyze the examples studied in \cite%
{[LLM3]} taking into account Theorem~\ref{mainc}. We observe that Theorem~%
\ref{mainc} gives an automated tool for proving the conjectures on the
values of $C^{s}$ given in \cite{[LLM3]}. Let $I_{k}$ be the closed interval
$I_{k}:=[\tilde{m}_{k}-\epsilon_{k},\tilde{m}_{k}+\epsilon_{k}]$ where $%
\tilde{m}_{k}$ and $\epsilon_{k}$ are defined in Theorem~\ref{mainc}.

\begin{enumerate}
\item \textbf{Cantor type sets in the real line.}

Let $\left\{ C_{r}\right\} _{r \in (0,\frac{1}{2})}$ be the family of linear
Cantor set defined by \eqref{canreal}. In \cite{[ZHZH]} it is proved that if
$0<r\leq \frac{1}{3}$, then
\begin{equation}
C^{s}(C_{r})=g_{3}(r)\text{ where }g_{3}(r):=2^{s}(1-r)^{s},
\label{centred1}
\end{equation}
and $s=-\frac{\log 2}{\log r}$ is the similarity dimension of $\{C_{r}\}$.

Table~\ref{Cantrtabc} records the results obtained from applying Theorem~\ref%
{mainc} in combination with Algorithm~\ref{algocent} to the family $%
\{C_{r}\} $. As in Section~\ref{packex}, the examples are chosen to
illustrate the behavior of $C^{s}(C_{r})$ with respect to \eqref{centred1}.

\begin{table}[]
\begin{center}
\begin{tabular}{|c|c|c|c|c|c|c|c|c|}
\hline
$\mathbf{r\leq \frac{1}{3}}$ & $\mathbf{q}$ & $\mathbf{s}$ & $\mathbf{%
k_{\max }}$ & $\mathbf{\mathcal{Q}}$ & $\mathbf{\tilde{m}_{k_{\max}}=
C^{s}(C_{r})=g_{3}(r)} $ & $\mathbf{\epsilon _{k_{\max }}}$ & $\mathbf{%
g_{3}(r)}$ & $\mathbf{k_{stb}}$ \\ \hline
$\frac{1}{4}$ & $1$ & $\frac{\log 2}{\log 4}$ & $20$ & $1.41421$ & $%
1.224744871392$ & $3.63798\times 10^{-12}$ & $1.224744871392$ & $2$ \\ \hline
$\frac{1}{3}$ & $1$ & $\frac{\log 2}{\log 3}$ & $20$ & $1.5$ & $%
1.199023144561$ & $1.68126\times 10^{-9}$ & $1.199023144561$ & $3$ \\ \hline
$\mathbf{r>\frac{1}{3}}$ &  &  &  &  & $\mathbf{\tilde{m}_{k_{\max}}}$, \ $%
\mathbf{g_{3}(r)\in I_{k_{\max }}}$ &  &  &  \\ \hline
$0.351$ & $2$ & $-\frac{\log 2}{\log 0.351}$ & $20$ & $1.50552$ & $%
1.188484857299$ & $1.01650\times 10^{-8}$ & $1.188484857299$ & $4$ \\ \hline
$0.3518$ & $2$ & $-\frac{\log 2}{\log 0.3518}$ & $20$ & $1.50562$ & $%
1.187959585122$ & $1.06731\times 10^{-8}$ & $1.187959585122$ & $4$ \\ \hline
&  &  &  &  & $\mathbf{\tilde{m}}_{k_{\max }},$ $g_{3}(r)\notin \emph{I}
_{k_{\max }}$ &  &  &  \\ \hline
$0.3519$ & $2$ & $-\frac{\log 2}{\log 0.3519}$ & $20$ & $1.50563$ & $%
1.187703097489$ & $1.07383\times 10^{-8}$ & $1.187893625780$ & $4$ \\ \hline
$0.4$ & $2$ & $-\frac{\log 2}{\log 0.4}$ & $20$ & $1.47986$ & $%
1.084545262462 $ & $1.66350\times 10^{-7}$ & $1.147884787390$ & $7$ \\ \hline
$0.45$ & $3$ & $-\frac{\log 2}{\log 0.45}$ & $21$ & $1.35502$ & $%
1.031518332488$ & $1.79220\times 10^{-6}$ & $1.086253162423$ & $19$ \\ \hline
\end{tabular}
\label{Cantrtabc}
\end{center}
\caption{Linear Cantor sets $\{C_{r}\}$}
\end{table}

As a consequence of these results we obtain the bounds:
\begin{eqnarray}
\mathbf{1.2247448713}87 &<&C^{s}(C_{\frac{1}{4}})<\mathbf{1.2247448713}96,
\notag \\
\mathbf{1.19902314}28 &<&C^{s}(C_{\frac{1}{3}})<\mathbf{1.19902314}63,
\notag \\
\mathbf{1.1884848}47 &<&C^{s}(C_{0.351})<\mathbf{1.1884848}68, \\
\mathbf{1.1879595}74 &<&C^{s}(C_{0.3518})<\mathbf{1.1879595}96,  \notag \\
\mathbf{1.187703}08 &<&C^{s}(C_{0.3519})<\mathbf{1.187703}11,  \label{CC35}
\\
\mathbf{1.084545}09 &<&C^{s}(C_{\frac{4}{10}})<\mathbf{1.084545}43,
\label{CC410} \\
\mathbf{1.0315}16 &<&C^{s}(C_{0.45})<\mathbf{1.0315}21.  \label{CC045}
\end{eqnarray}

The algorithm recovers the value given by \eqref{centred1} in the cases
where $r\leq 1/3$. To our knowledge, there is no general formula for $%
C^{s}(C_{r})$ when $r>\frac{1}{3}$, but in these cases we obtain estimates
of its value to accuracies of $7$, $6$, and $4$ decimal places, namely, $%
C^{s}(C_{0.351})\simeq 1.1884848$, $C^{s}(C_{0.3518})\simeq 1.1879595$,$%
C^{s}(C_{0.3519})\simeq 1.187703$, $C^{s}(C_{\frac{4}{10}})\simeq 1.084545$,
and $C^{s}(C_{0.45})\simeq 1.0315$. Moreover, as in Section~\ref{packex}, %
\eqref{CC35}, \eqref{CC410}, and \eqref{CC045} show that \eqref{centred1}
is, in general, not valid when $r>\frac{1}{3}$. In fact, $g_{3}(r)$ is an
upper bound for $C^{s}(C_{r})$.

\item \textbf{Sierpinski gaskets}. Let $S_{r}$ be the class of Sierpinski
gaskets defined by \eqref{sierclas} and let
\begin{equation*}
g_{4}(r):=\left[ 2(1-r)(r^{2}+r+1)^{\frac{1}{2}}\right]^{s}
\end{equation*}
where $s=-\frac{\log (3)}{\log (r)}$. The results obtained in \cite{[LLM3]}
for this class of Sierpinski gaskets led to the conjecture
\begin{equation}
C^{s}(S_{r})=g_{4}(r)\qquad \text{for all}\,\, r<0.25.  \label{packsierconj}
\end{equation}

In Table~\ref{siertc} we can see that, even when \eqref{packsierconj} was
conjectured for $r<0.25$, the algorithm output approximates the value given
by \eqref{packsierconj} in all the cases where $r\leq 0.277$ (with accuracy
of more than $12$ decimal places). For the cases where $r\geq 0.278$ we
observe the opposite behavior, $g_{4}(r)\notin \emph{I}_{k_{\max }}$, so %
\eqref{packsierconj} cannot hold, and $g_{4}(r)$ is an upper bound for $%
C^{s}(S_{r})$.
\begin{table}[]
\begin{center}
\begin{tabular}{|c|c|c|c|c|c|c|c|c|}
\hline
$\mathbf{r}$ & $\mathbf{s}$ & $\mathbf{q}$ & $\mathbf{k_{\max }}$ & $\mathbf{%
\mathcal{Q}}$ & $\mathbf{\tilde{m}_{k_{\max }}},$ $\mathbf{g_{4}(r)\in \emph{%
I}_{k_{\max }}}$ & $\mathbf{\epsilon_{k_{\max }}}$ & $\mathbf{g_{4}(r)}$ & $%
\mathbf{k_{stb}}$ \\ \hline
$1/27$ & $\frac{\log 3}{\log 27}$ & $1$ & $10$ & $1.05265$ & $1.252010347930$
& $1.28830\times 10^{-14}$ & $1.252010347930$ & $3$ \\ \hline
$0.2$ & $-\frac{\log 3}{\log 0.2}$ & $1$ & $12$ & $1.17602$ & $%
1.483264747602 $ & $3.16650\times 10^{-8}$ & $1.483264747602$ & $3$ \\ \hline
$1/4$ & $\frac{\log 3}{\log 4}$ & $1$ & $12$ & $1.15470$ & $1.535835728296$
& $5.66827\times 10^{-7}$ & $1.535835728296$ & $3$ \\ \hline
$0.277$ & $-\frac{\log 3}{\log 0.277}$ & $1$ & $12$ & $1.12349$ & $%
1.560819225967$ & $2.13042\times 10^{-6}$ & $1.560819225967$ & $>k_{\max }$
\\ \hline
&  &  &  &  & $\mathbf{\tilde{m}_{k_{\max }}},$ \ $\mathbf{g_{4}(r)\notin
I_{k_{\max }}}$ &  &  &  \\ \hline
$0.278$ & $-\frac{\log 3}{\log 0.278}$ & $1$ & $12$ & $1.12202$ & $%
1.561597393347$ & $2.23163\times 10^{-6}$ & $1.561690520340$ & $5$ \\ \hline
$1/3$ & $1$ & $1$ & $13$ & $1$ & $1.543702825201$ & $7.52671\times 10^{-6}$
& $1.602467233540$ & $>k_{\max }$ \\ \hline
$0.4$ & $-\frac{\log 3}{\log 0.4}$ & $2$ & $12$ & $1$ & $1.472023977311$ & $%
8.31250\times 10^{-4}$ & $1.624473448850$ & $>k_{\max }$ \\ \hline
\end{tabular}%
\end{center}
\caption{Sierpinski gaskets $\{S_{r}\}$}
\label{siertc}
\end{table}

Next we give the bounds provided by Theorem~\ref{mainc}:
\begin{eqnarray*}
\mathbf{1.2520103479303}3 &<&C^{s}(S_{\frac{1}{27}})<\mathbf{1.2520103479303}%
5, \\
\mathbf{1.4832647}15 &<&C^{s}(S_{\frac{2}{10}})<\mathbf{1.4832647}80, \\
\mathbf{1.53583}516 &<&C^{s}(S_{\frac{1}{4}})<\mathbf{1.53583}630, \\
\mathbf{1.5608}170 &<&C^{s}(S_{0.277})<\mathbf{1.5608}214, \\
\mathbf{1.56159}516 &<&C^{s}(S_{0.278})<\mathbf{1.56159}963, \\
\mathbf{1.543}695 &<&C^{1}(S_{\frac{1}{3}})<\mathbf{1.543}711, \\
\mathbf{1.47}119 &<&C^{s}(S_{\frac{4}{10}})<\mathbf{1.47}286.
\end{eqnarray*}

Observe that the above bounds suffice to prove the conjectural values
proposed in \cite{[LLM3]}. In particular, the bounds obtained for the case $%
C^{1}(S_{1/3})$ prove the conjecture $C^{1}(S_{1/3})\simeq 1.543$. We would
like to clarify that there was a minor error in \cite{[LLM3]} as, according
to the algorithm's output in this case, the correct conjectured value is $%
1.543$ and not $1.537$ as was written in \cite{[LLM3]}.

\item \textbf{Planar Cantor type sets }$K_{r}$.

Let $\left\{ K_{r}\right\} _{r\in (0,\frac{1}{2})}$ be the family of planar
Cantor type sets defined by \eqref{planrcantor} and $s=-\frac{\log 4}{\log r}
$. In \cite{[ZHZH1]} it is shown that
\begin{equation}
C^{s}(K_{r})=g_{5}(r)\text{ where }g_{5}(r):=\left( 2\sqrt{2}(1-r)\right)^{s}
\label{zhzhform}
\end{equation}
whenever $s\in (0,1)$, $(1-r)r^{\frac{2s-1}{1-s}}\geq 2$, and $\frac{3r^{s}}{
(1-r)^{s}}\leq 2^{-\frac{s}{2}}$.

These conditions hold for $r<r_{0}$ with $r_{0}\simeq 0.10832764$.

Table~\ref{canptc} shows the outcomes obtained applying Theorem~\ref{mainc}
together with Algorithm~\ref{algocent} to the family $K_{r}$. Observe that
the algorithm recovers the value given by \eqref{zhzhform} when $r<r_{0}$.
In the examples in Table~\ref{canptc} where \eqref{zhzhform} cannot be
rejected, $\mathbf{\tilde{m}}_{k}$ and $g_{5}(w)$ coincide to at least
twelve decimal places. For $r>0.17$, \eqref{zhzhform} does not hold true and
$g_{5}(r)$ is an upper bound for $C^{s}(K_{r})$.
\begin{table}[]
\begin{center}
\begin{tabular}{|c|c|c|c|c|c|c|c|c|}
\hline
$\mathbf{r<r_{0}}$ & $\mathbf{s}$ & $\mathbf{q}$ & $\mathbf{k_{\max}}$ & $%
\mathbf{\mathcal{Q}}$ & $\mathbf{\tilde{m}}_{k}=\mathbf{C}^{s}\mathbf{%
(K_{r})=g_{5}(r)}$ & $\mathbf{\epsilon_{k}}$ & $\mathbf{g_{5}(r)}$ & $%
\mathbf{k_{stb}}$ \\ \hline
$0.01$ & $\frac{\log 4}{\log 100}$ & $1$ & $7$ & $1.01422$ & $1.363372877653$
& $4.25566 \times 10^{-14}$ & $1.363372877653$ & $2$ \\ \hline
$0.05$ & $-\frac{\log 4}{\log 0.05}$ & $1$ & $10$ & $1.05824$ & $%
1.579962585475$ & $7.45665 \times 10^{-13}$ & $1.579962585475$ & $2$ \\
\hline
$0.1$ & $\frac{\log 4}{\log 10}$ & $1$ & $10$ & $1.09286$ & $1.755126487784$
& $1.12992 \times 10^{-9}$ & $1.755126487784$ & $2$ \\ \hline
$\mathbf{r\geq r}_{0}$ &  &  &  &  & $\mathbf{\tilde{m}}_{k},$ $\mathbf{%
g_{5}(r)\in I_{k_{\max }}}$ &  &  &  \\ \hline
$1/8$ & $\frac{\log 4}{\log 8}$ & $1$ & $10$ & $1.10064$ & $1.829652855011$
& $1.22729 \times 10^{-8}$ & $1.829652855011$ & $2$ \\ \hline
$0.16$ & $-\frac{\log 4}{\log 0.16}$ & $1$ & $10$ & $1.09847$ & $%
1.924421022097$ & $1.74625 \times 10^{-7}$ & $1.924421022097$ & $3$ \\ \hline
&  &  &  &  & $\mathbf{\tilde{m}_{k}}, \ \mathbf{g_{5}(r)\notin I_{k_{\max }}%
}$ &  &  &  \\ \hline
$0.17$ & $-\frac{\log 4}{\log 0.17}$ & $1$ & $10$ & $1.09465$ & $%
1.946542971745$ & $3.35957 \times 10^{-7}$ & $1.949655110042$ & $3$ \\ \hline
$0.2$ & $-\frac{\log 4}{\log 0.2}$ & $1$ & $10$ & $1.07339$ & $%
1.978683424694 $ & $1.94598 \times 10^{-6}$ & $2.020532989127$ & $3$ \\
\hline
$1/4$ & $1$ & $1$ & $10$ & $1$ & $1.954277821708$ & $2.157919 \times 10^{-5}$
& $2.121320343560$ & $9$ \\ \hline
$0.4$ & $-\frac{\log 4}{\log 0.4}$ & $3$ & $11$ & $1.19455$ & $%
1.650343901758 $ & $3.91609 \times 10^{-2}$ & $2.225958183662$ & $>k_{\max}$
\\ \hline
\end{tabular}%
\end{center}
\caption{Planar Cantor type sets $\left\{ K_{r}\right\}$}
\label{canptc}
\end{table}
\qquad \qquad

The bounds provided by Theorem~\ref{mainc} are:
\begin{eqnarray}
\mathbf{1.363372877652}81 &<&C^{s}(K_{\frac{1}{100}})<\mathbf{1.363372877652}%
91,  \notag \\
\mathbf{1.57996258547}383 &<&C^{s}(K_{0.05})<\mathbf{1.57996258547}533,
\notag \\
\mathbf{1.75512648}665 &<&C^{s}(K_{\frac{1}{10}})<\mathbf{1.75512648}892,
\notag \\
\mathbf{1.8296528}427 &<&C^{s}(K_{\frac{1}{8}})<\mathbf{1.8296528}673,
\notag \\
\mathbf{1.92442}084 &<&C^{s}(K_{0.16})<\mathbf{1.92442}120,  \notag \\
\mathbf{1.9465}426 &<&C^{s}(K_{0.17})<\mathbf{1.9465}434,  \notag \\
\mathbf{1.97868}147 &<&C^{s}(K_{0.2})<\mathbf{1.97868}538,  \notag \\
\mathbf{1.954}256 &<&C^{s}(K_{\frac{1}{4}})<\mathbf{1.954}300,  \notag \\
\mathbf{1.6}1118 &<&C^{s}(K_{0.4})<\mathbf{1.6}8951.  \label{can14bdd} \\
\end{eqnarray}

Finally, we remark that the rate of convergence given by Theorem~\ref{mainc}
provides an estimate of $C^{s}(K_{r})$ for the cases where we do not have a
general formula with accuracy that varies from seven decimal places to one
decimal places. Actually, in \cite{[LLM3]}, $1.95$ was proposed as a
conjectural value for $C^{s}(K_{\frac{1}{4}})$. Now, thanks to Theorem~\ref%
{mainc}, we have proved this conjecture, for, by \eqref{can14bdd} we have $%
C^{s}(K_{\frac{1}{4}})\simeq 1.954.$
\end{enumerate}

\subsection{The general case}

We present below Table \ref{gnrc} in order to illustrate the general case
(see Section \ref{general case packing}) for the centered Hausdorff measure.
It gives the results for all iterations in the computation of $%
C^{s}(C_{0.45})$ with $C_{0.45}$ being the central Cantor set in the line
with contraction ratio $0.45$.

The columns in the table are: the iteration $k$, the center $x$ and the end
point $y$ of the optimal ball, the radius $d=\left\vert x-y\right\vert $ of
the optimal ball, the estimate $\tilde{m}_{k}$ of $C^{s}(C_{0.45})$ at the $%
k $th iteration, and the interval $I_{k}$ to which we can be sure that $%
C^{s}(C_{0.45})$ belongs. Again, for simplicity all the values reported in
the table are rounded to eight decimal places.
\begin{table}[]
\begin{center}
\begin{tabular}{|c|c|c|c|c|c|}
\hline
$\mathbf{k}$ & $\mathbf{x}$ & $\mathbf{y}$ & $\mathbf{d}$ & $\mathbf{\tilde{m%
}}_{k}$ & $\mathbf{I}_{k}$ \\ \hline
$5$ & $0.55$ & $0.091125$ & $0.458875$ & $1.02422358$ & $%
(0.39037468,1.65807248)$ \\ \hline
$6$ & $0.55$ & $0.111375$ & $0.438625$ & $1.03859290$ & $(0.75336089,\
1.32382491)$ \\ \hline
$7$ & $0.56014905$ & $0.11967877$ & $0.44047028$ & $1.03299380$ & $%
(0.90463940,1.16134821)$ \\ \hline
$8$ & $0.44626331$ & $0.00373669$ & $0.44252661$ & $1.03252769$ & $%
(0.97476821,\ 1.09028718)$ \\ \hline
$9$ & $0.55662225$ & $0.11305651$ & $0.44356574$ & $1.03231740$ & $%
(1.00632562,\ 1.05830917)$ \\ \hline
$10$ & $0.55549190$ & $1$. & $0.44450810$ & $1.03191238$ & $(1.02021608,\
1.04360868)$ \\ \hline
$11$ & $0.55549190$ & $0.99965949$ & $0.44416759$ & $1.03180195$ & $%
(1.02653862,\ 1.03706529)$ \\ \hline
$12$ & $0.55567918$ & $1.00$ & $0.44432082$ & $1.03153497$ & $(1.02916647,\
1.03390348)$ \\ \hline
$13$ & $0.55567918$ & $1$ & $0.44432082$ & $1.03153497$ & $(1.03046914,\
1.03260080)$ \\ \hline
$14$ & $0.55567918$ & $1.0$ & $0.44432082$ & $1.03153497$ & $(1.03105535,\
1.03201460)$ \\ \hline
$15$ & $0.44430686$ & $0.888625$ & $0.44431814$ & $1.03152958$ & $%
(1.03131375,\ 1.03174542)$ \\ \hline
$16$ & $0.55568686$ & $1$. & $0.44431314$ & $1.03151950$ & $(1.03142237,\
1.03161663)$ \\ \hline
$17$ & $0.55568686$ & $1$ & $0.44431314$ & $1.03151950$ & $(1.03147579,\
1.03156321)$ \\ \hline
$18$ & $0.44431187$ & $0.888625$ & $0.44431313$ & $1.03151949$ & $%
(1.03149981,\ 1.03153916)$ \\ \hline
$19$ & $0.44431244$ & $0.888625$ & $0.44431256$ & $1.03151833$ & $%
(1.031509482,1.03152719)$ \\ \hline
$20$ & $0.44431244$ & $0.888625$ & $0.44431256$ & $1.03151833$ & $%
(1.03151434,1.03152232)$ \\ \hline
$21$ & $0.44431244$ & $0.888625$ & $0.44431256$ & $1.03151833$ & $%
(1.03151654,1.03152013)$ \\ \hline
\end{tabular}%
\end{center}
\caption{ Central Cantor set in the line $C_{0.45}$ }
\label{gnrc}
\end{table}

One can see that the values continue changing up to the limit of our
computational power.

\section{Conclusions\protect\bigskip \label{conclusions}}

Research on the computability of metric measures on self-similar sets
started in \cite{[MO]}, ten years ago. The general method for the
computation of metric measures was established in that paper, and a
discussion on the computability of the existing metric measures was
initiated. In the introduction of that paper one can read:


``The exhaustive class of coverings used by this (the Hausdorff) measure
gives it a special place as the smallest among all measures based on
coverings. The price to be paid for such privilege is that the research on
the exact Hausdorff measure of a self-similar set leads, with few
exceptions, to the computation of bounds, but the exact Hausdorff measure of
a self-similar set will remain unknown for some time.''

In regard to the computation of the packing measure, there is written in
\cite{[MO]}:

``Another consequence of these results is that the packing measure could be
easier than the Hausdorff measure from a computational point of view . . .
since the search for sets of optimal density is restricted to balls centered
at $E$.''

These predictions have been confirmed in the subsequent literature gathered
in the bibliographical references below. Some results, already discussed
earlier, on the exact packing and centered Hausdorff measures have been
obtained using geometric methods, but the list of known results on Hausdorff
measure has not grown significantly. The predictions in \cite{[MO]} have
also been confirmed by the work on the computability of metric measures
reported in \cite{[LLM1]}, \cite{[LLM2]}, \cite{[LLM3]}, \cite{[LLM4]}, and,
in particular, by the results in the present paper, that permit a more
detailed discussion of the frontiers of the computability of metric
measures. Furthermore, something new can be added to the early expectations:
by results in \cite{[LLM2]}, \cite{[LLM3]}, and the present paper, the
centered Hausdorff measure and the packing measure can be added to the list
of ``computable'' metric measures. This is fortunate, since a covering and a
packing based measure are available for computation, at least for
self-similar sets satisfying the SSC. As pointed out in Remark~\ref{rough
idea}, this also gives additional valuable information on the spectrum of
densities of $\mu$.

The above conclusions motivate a call to revisit well-established folklore
on the topic of metric measures. Because of the double step definitions
required for the packing and centered Hausdorff measures, it is a common
opinion that these measures are too awkward to handle. However, in the
setting of self-similar sets satisfying the OSC (and so also the SSC), the
second step in the definitions of $P^{s}$ and $C^{s}$ can be omitted (and,
in the case of the packing measure, it can be omitted also in the more
general setting of compact sets with finite packing measure). As shown in
the present paper, one can take advantage of this fact for computational
purposes, and it is now clear that the measures $P^{s}$ and $C^{s}$ will
play a relevant role in the future, at least in computational issues.

We now discuss the frontiers of the computability of metric measures. At the
present, two metric measures, $C^{s}$ and $P^{s}$, can be computed with the
accuracy necessary for potential technical applications only in the case of
self-similar sets satisfying the SSC and having small contraction ratios. In
these cases, if the contraction ratios are small enough the results in this
paper indicate that an optimal ball might exist, and a formula giving the
exact value of the corresponding metric measure might be found by
theoretical methods. Moreover, our algorithm is an efficient tool for
identifying what might be the optimal ball, if the algorithm stabilizes at
an early iteration. On the other hand, if the contraction ratios are large,
then any attempted theoretical approach could be doomed to failure, since
the problem of the calculation of $C^{s}$ and $P^{s}$ is essentially
computational.

The bounds on the maximum error provided by Theorems~\ref{main} and \ref%
{mainc} decrease exponentially with the number $k$ of iterations, but the
number of calculations grows at a much faster rate when $k$ increases, since
the number of feasible balls depends on the square of the number of points
in $A_{k}$, and this in turn grows exponentially with $k$. In this regard,
the reductions, obtained in \cite{[LLM3]} and \cite{[LLM4]}, of the families
of balls that need to be considered are crucial. However, if the contraction
radii increase, two things occur that can render accurate computation of $%
P^{s}$ and $C^{s}$ impossible. First it is necessary to use smaller balls,
since the minimum separation distance $c$ between the basic cylinders
decreases, so the number of balls to be explored increases. Second and more
important, in order to obtain good estimates of the $\mu$-measure of a ball,
it is necessary to go to an iteration $k$ for which the size of the
cylinders of the $k$th generation is small relative to the size of the ball,
and for this purpose we have to go to more advanced iterations for balls of
small size. On the other hand, if the contraction ratios are large, $k$ must
be taken to be still larger, so that the $k$-cylinders will be sufficiently
small. In this case the computability of $P^{s}$ and $C^{s}$ encounters
severe obstacles (see Remark~\ref{timeq}).

Although the theoretical methods proposed in \cite{[MO]} can, in principle,
be applied to self-similar sets satisfying the OSC, these cases are not yet
amenable to computation by our method. This is because the OSC can be viewed
as a limiting case, with the size of feasible balls going to zero (see \cite%
{[Ay]} for an example in the line with the OSC, where the exploration of
arbitrarily small balls must be undertaken in the search for a ball with
maximal density) and the argument given above for cases satisfying the SSC
and having large contraction ratios applies in an extreme form.

The computation of the spherical Hausdorff measure $H_{sph}^{s}$ could still
be feasible for the easier cases of small contraction ratios, but the
results will be much poorer than those for $P^{s}$ or $C^{s}$ for increasing
contraction ratios, because the class of balls centered at arbitrary points
in the ambient space, which is the covering class used in the definition of $%
H_{sph}^{s}$, is a much larger class than the covering classes used in the
definitions of $P^{s}$ and $C^{s}$. Moreover, the computation of the
Hausdorff measure is still unreachable by the argument given in \cite{[MO]}.

\begin{acknowledgement}
The computational part of this work was performed in EOLO, the HPC of
Climate Change of the International Campus of Excellence of Moncloa, funded
by the MECD and MICINN. This is a contribution to the CEI Moncloa.

Part of the present work has been done while Marta Llorente was visiting
Prof. Claude Tricot at the Mathematics department of the University Blaise
Pascal. The author is grateful to Prof. Claude Tricot for valuable comments,
fruitful discussions and hospitality.
\end{acknowledgement}

.

\end{document}